\documentclass[preprint,10pt]{elsarticle}
 \usepackage{graphicx}
\usepackage{amssymb}
 \usepackage{amsthm}
 \theoremstyle{plain}
 \newtheorem{theorem}{Theorem}[section]

 \newtheorem{lemma}{Lemma}[section]
 \newtheorem{proposition}{Proposition}[section]
 
\theoremstyle{definition}
 \newtheorem{definition}{Definition}[section]
 \newtheorem{example}{Example}[section]
 \newtheorem{remark}{Remark}
\theoremstyle{remark}
\usepackage[english]{babel}
\usepackage{amssymb,amstext,amsmath}
\usepackage{amsfonts,amsmath,amssymb,amscd}
\usepackage{amsfonts}
\usepackage{mathrsfs}
 \usepackage{lineno}

\journal{arXiv}

\begin{document}

\begin{frontmatter}
 \title{Submartingale property of set-valued stochastic integration associated with Poisson process and related integral equations on Banach spaces\tnoteref{label0}}
 \tnotetext[label0]{This a revised version of the manuscript posted in arXiv with Id number 2002.09220. This work is partly supported by Beijing National Science Foundation (1192015), the Construct Program of the Key Discipline in Hunan Province and State Scholarship Fund of CSC.}
\author[label1]{Jinping Zhang}
 \address[label1]{Department of Mathematics and Physics, North China Electric Power University,
Beijing, 102206, P.R.China}
\ead{zhangjinping@ncepu.edu.cn}
\author[label3]{Itaru Mitoma}
\address[label3]{Department of Mathematics, Saga University, Saga, 840-8502, Japan}
\ead{mitoma@ms.saga-u.ac.jp}
\author[label4]{Yoshiaki Okazaki}
\address[label4]{Department of Systems Design and Informatics, Kyushu
Institute of Technology, Iizuka, 820-8502, Japan}
\ead{okazaki@flsi.or.jp}
\begin{abstract}
In an M-type 2 Banach space, firstly we explore some properties of
the set-valued stochastic integral associated with the stationary Poisson point
process. By using the Hahn decomposition theorem and bounded linear functional, we obtain the main result: the integral of a set-valued stochastic process with respect to the compensated Poisson measure is a set-valued submartingale but not a martingale unless the integrand degenerates into a single-valued process. Secondly we study the strong solution to the set-valued
stochastic integral equation, which includes a set-valued drift,
a single-valued diffusion driven by a Brownian motion and the
set-valued jump driven by a Poisson process.

\end{abstract}


\end{frontmatter}


\section{Introduction}
\label{author_sec:1}
Set-valued stochastic calculus is the natural extension of single-valued case.  Aumann \cite{Aum}(1965) defined the expectation of set-valued random variables. Hiai and Umegaki \cite{Hia} (1977) gave the definition of set-valued conditional expectation, set-valued martingale (or super/submartingale). After that research on set-valued stochastic integral and differential equation (or inclusion) has been received much attention.

 Kisielewicz \cite{Kis} (1997) studied the stochastic integral of set-valued
stochastic process with respect to Brownian motion in $d$-dimensional Euclidean space $\mathbb R^{d}$, where the integral is defined as a subset of $L^{2}(\Omega; \mathbb {R}^{d})$, which is called the trajectory integral. Kim and Kim \cite{Kim2}, Jun and Kim \cite{Jun}  defined the set-valued It$\hat{o}$ integral (different from the trajectory integral) with respect to Brownian motion by using an indirect method such that the integral is a set-valued stochastic process. After that, there are a lot of research  related to set-valued stochastic integrals. For example, Li and Ren \cite{Li2} considered the integral as a mapping from the product space $[0,+\infty)\times \Omega$ to the power set of $\mathbb{R}^{d}$, where the measurability is also considered in the sense of product $\sigma$-algebra generated by $[0,+\infty)\times \Omega$. Michta \cite{M.M} studied both set-valued integral and trajectory integral with respect to semimartingale with finite path variation in $\mathbb{R}^{d}$. In an M-type Banach space $\frak X$, Zhang et al. \cite{Zha1,2012,Zha4} considered the set-valued integrals with respect to Brownian motion, martingales and Poisson point processes respectively.

There are two ways to extend single-valued stochastic differential equation. One is differential inclusions (also being called multi-valued differential equations in some references). For example, the nice references written by J. Ren and J. Wu \cite{ren20091}, J. Ren et al. \cite{ren2010} studied stochastic differential inclusion with single-valued Brownian diffusion in $\mathbb{R}^{d}$. J. Ren and J. Wu \cite{ren2011} studied the differential inclusion with Brownian diffusion and Poisson jump in $\mathbb{R}^{d}$ as follows:
\begin{equation}\label{inclusion}
dX_{t}\in  -A(X_{t})dt+b(X_{t})dt+\sigma(X_{t})dB_{t}
 +\int_{Z_{0}}f(X_{t-},z)\tilde{N}(dtdz)+\int_{Z/Z_{0}}g(X_{t-},z)N(dtdz),
\end{equation}
where $A$ is a set-valued operator. Other mappings are single-valued.

Another way to extend the single-valued stochastic equation is to turn the inclusion `$\in$' in \eqref{inclusion} into an equality `$=$'. Here we call it a set-valued integral (or differential) equation. That is to say, considering the  solution $X(t)$  as a set-valued process.
     In $\mathbb{R}^{d}$ space, there are some references about set-valued differential equation without jump, e.g. \cite{Li2010,Ogura} etc.  In \cite{Mic}, the authors studied the set-valued equation driven by martingale, where the set-valued integral is the trajectory integral. In an M-type 2 Banach space, Zhang et al \cite{Zha3}, Mitoma et al \cite{Mitoma} explored the strong solutions to set-valued stochastic differential equations, where the diffusion part is single-valued since the set-valued integral with respect to Brownian motion may be unbounded a.s.

   The Poisson point process is a special kind of L$\acute{e}$vy process with a wide range of applications. It is important
 in both random mathematics (see e.g. \cite{Dett,Ike,Kunita2})
and applied fields (see e.g.\cite{Kunita2}). For convenience , we consider the stationary
Poisson process $\bf p$ with a finite characteristic measure $\nu$. Both of the Poisson random
measure $N(dsdz)$ (where $z\in Z$, the state space of $\bf p$) and
the compensated Poisson random measure $\tilde{N}(dsdz)$ are of
finite variation a.s., which is different from Brownian motion. Based on the work \cite{2012} and  \cite{Zha3}, in an M-type 2 Banach space $\frak X$, by using the Hahn decomposition theorem of a space and properties of the bounded linear functional, we shall prove that stochastic integrals
of set-valued predictable  processes with respect to $N(dsdz)$ and
$\tilde{N}(dsdz)$ are $L^{2}$-integrably bounded. The integral with respect to the compensated measure is a submartingale.  The last theorem (Theorem 3.7) in \cite{2012} states that the integral is a set-valued martingale. But unfortunately the integral is not a martingale unless the integrand degenerates into a single-valued stochastic process a.s. See Theorem \ref{thm:unnmartingale} in this paper.

Thanks to the integrable boundeness of set-valued stochastic integral with respect to Poisson point process with finite characteristic measure,  based on the work \cite{Mitoma,Zha3},  we can study the extended set-valued stochastic integral equations with set-valued Poisson jump and single-valued Brownian motion diffusion as follows:
 \begin{equation}\label{equation111}
X_{t}=cl\big\{X_{0}+\int_{0}^{t}a(s,X_{s})ds+\int_{0}^{t}b(s,X_{s})dB_{s}
+\int_{0}^{t+}\int_{Z}c(s,z,X_{s-})N(dzds)\big\},
\end{equation}
for $t\in[0,T]$ a.s. where $a(\cdot,\cdot)$ and $c(\cdot,\cdot)$ are set-valued and $b(\cdot,\cdot)$ is single-valued. $\{B_{t};t\geq 0\}$ is a real valued Brownian motion. The notation $cl$ stands for the closure in the Banach space $\frak X$.

This paper is organized as follows: Section \ref{author_sec:2} is about basic notations and auxiliary results related to the set-valued theory.
In Section \ref{author_sec:3}, firstly we review the stochastic
integrals for $\frak X$-valued $\mathscr{S}$-predictable processes
with respect to $N(dsdz)$ and $\tilde{N}(dsdz)$ as required later. Then we study the
stochastic integrals for set-valued $\mathscr{S}$-predictable
processes with respect to $N(dsdz)$ and $\tilde{N}(dsdz)$. Section \ref{author_sec:4}
  is about the existence and uniqueness of strong solution to  equation \eqref{equation111}. Section \ref{author_sec:5} is a concluding remark.

\section{Preliminaries}
\label{author_sec:2}
  Let $(\Omega, {\mathcal F}, \{\mathcal F_{t}\}_{t\geq 0}, P)$ be a filtered complete probability
space, in which the filtration $\{\mathcal F_{t}\}_{t\geq 0}$  satisfying the
usual condition. Let $\mathcal B(E)$ be the Borel field of a topological
space $E$, $(\frak X,\| \cdot \|)$ a real separable Banach space equipped
with the norm $\|\cdot\|$ and $\bf K(\frak X)$ ($\bf
K_{b}(\frak X)$, $\bf K_{c}(\frak X)$) the family of all nonempty
closed (resp. closed bounded, closed convex) subsets of $\frak X$.
 Let $1\leq p<+\infty$ and
 $L^{p}(\Omega, \mathcal F, P; \frak X)$
 (denoted briefly by $L^{p}(\Omega; \frak X)$)
 be the Banach space of equivalence classes of
$\frak X$-valued $\mathcal F$-measurable functions $f:
\Omega\rightarrow \frak X$ such that the norm
$ \|f\|_{p}=\Big\{\int_{\Omega}\|f(\omega)\|^{p}dP\Big\}^{1/p}$
is finite. An $\frak X$-valued function $f$ is called {\em
$L^{p}$-integrable} if $f\in L^{p}(\Omega; \frak X)$.

 A set-valued function $F: \Omega\rightarrow {\bf{K}}(\frak X)$ is said to be {\em measurable} if for any open set $O\subset \frak X$,
the inverse $F^{-1}(O):=\{\omega\in\Omega: F(\omega)\cap
O\neq\emptyset\}$ belongs to $\mathcal F$. Such a function $F$ is
called a {\em set-valued random variable}. Let $\mathcal
M\big(\Omega, \mathcal F, P; {\bf K}(\frak X)\big)$ be the family of
all set-valued random variables, which is briefly denoted by
$\mathcal M\big(\Omega; {\bf K}(\frak X)\big)$.

 For
any open subset $O\subset \frak X$, set
$Z_{O}:=\{E\in {\bf K}(\frak X): E\cap O\neq\emptyset\},$ and
 $\mathcal C:=\{Z_{O}: O\subset\frak X ,\ O \ is\ open\},$
 and let $\sigma(\mathcal C)$ be
the $\sigma$-algebra generated by $\mathcal C$.
 A set-valued function $F: \Omega\rightarrow \bf K(\frak X)$ is
measurable if and only if $F$ is $\mathcal F/ \sigma(\mathcal
C)$-measurable. By Kuratowski-Ryll-Nardzewski Selection Theorem (see e.g. \cite {Cha},
 page 509), any set-valued random variable  $F: \Omega \rightarrow {\bf K}(\frak X)$ admits an measurable selection $f$ such that $f(\omega)\in F(\omega)$ for each $\omega\in \Omega$.

For $A, B\in 2^{\frak X}$ (the power set of $\frak X$), $H(A, B)\geq
0$ is defined by
$$
H(A,B):= \max \{\sup_{x\in A}\inf_{y\in B}||x-y||, \sup_{y \in
B}\inf_{x\in A}||x-y||\},
$$
which is called the {\em
Hausdorff metric}. It is well-known that $\bf K_{b}(\frak X)$
equipped with the metric $H$ denoted by ($\big({\bf K}_{b}(\frak X),
H\big)$) is a complete metric space.

For $F\in \mathcal M\big(\Omega, {\bf K}(\frak X)\big)$, the family
of all $L^{p}$-integrable selections is defined by
$$
S^{p}_{F}(\mathcal F):=\{f\in L^{p}(\Omega, {\mathcal F}, P; \frak
X) : f (\omega)\in F(\omega) \
 a.s.\}.
$$
In the following, $S_{F}^{p}(\mathcal F)$ is denoted briefly by
$S_{F}^{p}$. If $S_{F}^{p}$ is nonempty, $F$ is said to be {\em
$L^{p}$-integrable}. $F$ is called {\em $L^{p}$-integrably bounded}
if there exits a function $h\in L^{p} (\Omega, {\mathcal F}, P;
\mathbb R)$ such that $\|x\|\leq h (\omega)$ for any $x$ and
$\omega$ with $x\in F(\omega)$. It is equivalent to that $\|F\|_{\bf
K}\in L^{p}(\Omega; \mathbb R)$, where $\|F(\omega)\|_{\bf
K}:=\sup\limits_{a\in F(\omega)}{\|a\|}$.
 The family of all measurable $\bf K(\frak X)$-valued  $L^{p}$-integrably bounded functions
 is denoted by $L^{p}\big(\Omega, \mathcal F, P; {\bf K}(\frak X)\big)$. Write it for
brevity as $L^{p}\big(\Omega; {\bf K} (\frak X)\big)$.

The {\em integral (or expectation)} of a set-valued random variable
$F$ was defined by Aumann in 1965 (\cite{Aum}):
\begin{equation*}
E[F]:=\{E[f]: f\in S_{F}^{1}\}.
\end{equation*}

Let $F\in \mathcal M(\Omega; \frak X)$, $1\leq p<+\infty$. Then $F$
is $L^{p}$-integrably bounded if and only if  $S_{F}^{p}$ is
nonempty and bounded in $L^{p}(\Omega;\frak X)$ (see e.g. \cite{Zha3} ).

Let $F_{1}, F_{2}\in \mathcal M(\Omega; \frak X)$
and $F(\omega)=cl(F_{1}(\omega)+F_{2}(\omega))$ for all
$\omega\in\Omega$. Then $F\in \mathcal M(\Omega; \frak X)$. Moreover
if $S_{F_{1}}^{p}$ and $S_{F_{2}}^{p}$ are nonempty where $1\leq
p<\infty$, then $S_{F}^{p}=cl(S_{F_{1}}^{p}+S_{F_{2}}^{p})$, the
closure in $L^{p}(\Omega; \frak X)$ (see e.g.\cite{Hia}).

Let $\mathbb R_{+}$ be the set of all nonnegative real numbers and
$\mathcal B_{+}:=\mathcal B(\mathbb R_ {+})$. $\mathbb{N}$ denotes
the set of natural numbers.
 An $\frak X$-valued stochastic process $f=\{f_{t}: t\geq0\}$ (or denoted
by $f=\{f(t): t\geq0\}$) is defined as a function $f: \mathbb
R_{+}\times\Omega \longrightarrow\frak X$ with the $\mathcal
F$-measurable section $f_{t}$, for $t\geq0$. We say $f$ is {\em
measurable }if $f$ is $\mathcal B_{+}\otimes\mathcal F$-measurable.
The process $f=\{f_{t}: t\geq0\}$ is called {\em $\mathcal
F_{t}$-adapted} if $f_{t}$ is $\mathcal F_{t}$-measurable for every
$t\geq0$. $f=\{f_{t}: t\geq0\}$ is called {\em predictable} if it is
$\mathcal P$-measurable, where $\mathcal P$ is the $\sigma$-algebra
generated by all left continuous and $\mathcal F_{t}$-adapted
stochastic processes.

In a fashion similar to the $\frak X$-valued stochastic process, a
{\em set-valued stochastic process} $F=\{F_{t}: t\geq 0\}$ is
defined as a set-valued function $F: \mathbb R_{+}\times
\Omega\longrightarrow \bf K(\frak X)$ with $\mathcal F$-measurable
section $F_{t}$ for $t\geq 0$. It is called {\em measurable} if it
is $\mathcal B_{+}\otimes \mathcal F$-measurable, and {\em
${\mathcal F}_{t}$-adapted} if for any fixed $t$, $F_{t}(\cdot)$ is
${\mathcal F}_{t}$-measurable. $F=\{F_{t}: t\geq 0\}$ is called {\em
predictable} if it is $\mathcal P$-measurable.

\begin{definition}\label{def:001}
(see \cite{Hia}) An integrable bounded convex set-valued
${\mathcal{F}}_{t}$-adapted stochastic process $\{F_{t},{\mathcal{F}}_{t}:
t\geq 0 \}$ is called a {\em set-valued ${\mathcal{F}}_{t}$-martingale}
if for any $0\leq s \leq t$ it holds that
$E[F_{t}|{\mathcal{F}}_{s}]=F_{s}$ in the sense of
$S_{E[F_{t}|{\mathcal{F}}_{s}]}^{1}({\mathcal{F}}_{s})=S_
{F_{s}}^{1}({\mathcal{F}}_{s})$.

It is called a {\em set-valued submartingale (supermartingale)} if
for any $0\leq s\leq t$, $E[F_{t}|{\mathcal{F}}_ {s}]\supset F_{s}$
(resp. $E[F_{t}|{\mathcal{F}}_{s}]\subset F_{s} $) in the sense of
$S_{E[F_{t}|{\mathcal{F}}_{s}]}^{1}({\mathcal{F}}_{s}) \supset S_
{F_{s}}^{1}({\mathcal{F}}_{s})$ (resp.
$S_{E[F_{t}|{\mathcal{F}}_{s}]}^{1}({\mathcal{F}}_{s}) \subset S_
{F_{s}}^{1}({\mathcal{F}}_{s})$)

\end{definition}

Note: This is the original definition of set-valued martingale given
by Hiai and Umegaki (1977) in \cite{Hia}. There are some references
which give the definition without the assumptions of convexity or
integrably boundedness (only assume it is integrable and $\mathcal
{F}_{t}$-adapted), see e.g. \cite{Li1}. In this paper, we use the
original definition.

 An $\frak X$-valued martingale $f=\{f_{t}, {\mathcal{F}}_{t}, t
\geq 0\}$ is called an {\em ${\bf L}^{p}$-martingale selection } of
the set-valued stochastic process $F=\{F_{t}, {\mathcal{F}}_{t}, t \geq
0\}$ if it is an $L^{p}$-selection of $F=\{F_{t}, {\mathcal {F}}_{t},
t\geq 0\}$. The family of all $L^{p}$-martingale selections of
$F=\{F_{t}, {\mathcal{F}}_{t}: t\geq 0\}$ is denoted by ${\bf
MS}^{p}(F(\cdot))$.
 Briefly, write $\bf{MS(F)} = {\bf MS}^{1}(F(\cdot))$.

 For interval-valued martingale, here we list a known result, which will be used later.

\begin{theorem}(\cite{Zha4})\label{pro:003}
Let $F=\{F(t), {\mathcal{F}}_{t}: t\geq 0\}$ be an adapted
interval-valued
 stochastic process, and
 $F=\{F(t), \mathcal{F}_{t}: t\geq 0\}\subset L^{1}(\Omega,
\mathcal{F}, P; {\bf {K}}_{c}(\mathbb{R})),$
 then the following two statements are equivalent:

 (1) $F=\{F(t), \mathcal{F}_{t}: t\geq 0\}$ is an interval-valued
 martingale;

 (2) there exist two real-valued martingale selection $\xi=
\{\xi(t), \mathcal{F}_{t}: t\geq
 0\}$ and $\eta=\{\eta(t), \mathcal{F}_{t}: t\geq 0\}$, s.t.
for each
 $t$, $F(t,\omega)=[\xi(t,\omega), \eta(t,\omega)]$ a.s.

\end{theorem}

\section{Properties of set-valued integration associated with Poisson processes}
\label{author_sec:3} In this section, in an M-type 2 Banach space, at first we will briefly
review the stochastic integrals with respect to the Poisson random
measure and the compensated Poisson random measure for $\frak
X$-valued and $\bf {K}(\frak X)$-valued stochastic processes, which
are studied in \cite{2012}. Then we study some other
properties of stochastic integrals for $\bf {K} (\frak X)$-valued
stochastic processes, such as the $L^{2}$-integrable boundedness,
set-valued submartingale property etc.

\subsection{Single-valued stochastic integrals w.r.t. Poisson processes}

The following  definitions and notations related to Poisson point processes come from
 \cite{Ike} and \cite{Watanabe}.

Let $\frak X$ be a separable Banach space and $Z$ be another
separable Banach space with $\sigma$-algebra $\mathcal{B}(Z)$. A
{\em point function ${\bf p}$} on $Z$ means a mapping ${\bf{p}}:
{\bf D_{p}}\rightarrow Z$, where the domain ${\bf {D_{p}}}$ is a
countable subset of $[0, T]$. ${\bf {p}}$ defines a counting measure
${N_{\bf p}}(dtdz)$ on $[0, T]\times Z$ (with the product
$\sigma$-algebra $\mathcal{B}([0, T])\otimes\mathcal{B}(Z)$) by
\begin{equation}
{N_{\bf p}}((0, t], U):=\#\{\tau\in {\bf{D_{p}}}: \tau\leq t,
{\bf{p}}(\tau)\in U\}, \ t\in(0,T],\ U\in \mathcal{B}(Z).
\end{equation}
For $ 0\leq s<t\leq T$,
${N_{\bf p}}((s, t], U):=N_{\bf p}((0,t], U)-N_{\bf p}((0,s], U).$
In the following, we also write ${N_{\bf p}}((0, t], U)$ as ${N_{\bf
p}}(t, U)$.

 A {\em point process} (denoted by $\bf {p}:=(\bf p)_{t\geq 0}$) is obtained by randomizing the notion of
point functions. If there is a continuous $\mathcal{F}_{t}$-adapted
increasing process $\hat{N}_{\bf
 p}$ such that for $U\in \mathcal{B}(Z)$ and $t\in[0, T]$,
 $\tilde{N}_{\bf p}(t, U):=N_{\bf p}(t, U)-\hat{N}_{\bf p}(t,
 U)$ is an $\mathcal{F}_{t}$-martingale,
 then the random measure $\{\hat{N}_{\bf p}(t, U)\}$ is called the {\em
 compensator} of the point process ${\bf p}$ (or $\{N_{\bf p}(t,
 U)\}$) and the process $\{\tilde {N}_{\bf p}(t, U)\}$ is called the {\em compensated} point process.

A point process ${\bf p}$ is called the {\em Poisson Point Process}
if $N_{\bf p}(dtdz)$ is a Poisson random measure on $[0, T]\times
Z$. A Poisson point process is stationary if and only if its
intensity measure $\nu_{\bf p}(dtdz)=E[N_{\bf p}(dtdz)]$ is of the
form
$\nu_{\bf p}(dtdz)=dt\nu(dz)$
for some measure $\nu(dz)$ on $(Z, \mathcal{B}(Z))$. $\nu(dz)$ is
called the {\em characteristic measure of  \ ${\bf p}$}.

Let $\nu$ be a $\sigma$- finite measure on $(Z, \mathcal{B}(Z))$,
(i.e. there exists $U_{i}\in \mathcal{B}(Z), i\in \mathbb{N}$,
pairwise disjoint such that $\nu(U_{i})<\infty$ for all $i\in
\mathbb{N}$ and $Z=\cup_{i=1}^{\infty}U_{i}$),
 ${\bf p}=({\bf
p)}_{t\geq 0}$ be
 the $\mathcal {F}_{t}$-adapted stationary Poisson point process on $Z$ with the characteristic
 measure $\nu$ such that the compensator $\hat{N}_{\bf p}(t, U)=E[N_{\bf p}(t,
 U)]=t\nu(U)$ (non-random).

 For convenience, we will omit the subscript $\bf p$ in the above
notations and assume $\nu(Z)$ is finite.

For any $U\in \mathcal{B}(Z)$, both
$\{N(t, U), t\in[0,T]\}$ and $\{\tilde{N}(t, U), t\in[0,T]\}$ are
stochastic processes with finite variation a.s.
For convenience, from now on, we suppose $\nu$ is a finite measure
in the measurable space $(Z, \mathcal{B}(Z))$.

An $\frak X$-valued mapping $f$ defined on $[0, T]\times
Z\times \Omega$ is called $\mathscr{S}$-predictable if the mapping
$(t, z, \omega)\rightarrow f(t, z, \omega)$ is
$\mathscr{S}/\mathcal{B}(\frak X)$-measurable, where $\mathscr{S}$
is the smallest $\sigma$-algebra
with respect to which all mappings $g: [0, T]\times Z\times
\Omega\rightarrow \frak X$ satisfying (i) and (ii) below are
measurable:

(i) for each $t\in[0, T]$, the mapping $(z,\omega)\rightarrow
g(t,z,\omega)$ is $\mathcal{B}(Z)\otimes
\mathcal{F}_{t}$-measurable;

(ii) for each $(z,\omega)\in Z\times \Omega$, the mapping
$t\rightarrow g(t,z,\omega)$ is left continuous.

\begin{remark}
(see e.g. \cite{Watanabe}) $\mathscr{S}=\mathcal{P}\otimes
\mathcal{B}(Z)$, where $\mathcal{P}$ denotes the $\sigma$-field on
$[0, t]\times \Omega$ generated by all left continuous and
$\mathcal{F}_{t}$-adapted processes.
\end{remark}

Set

$\mathscr{L}=\Big\{f:  f \ is\ \mathscr{S}{\rm-}predictable \ and
\ E\Big[\int_{0}^{T}\int_{Z}\|f(t,z,\omega)\|^{2}\nu(dz)dt\Big]<\infty\Big\}$
equipped with the norm
$\|f\|_{\mathscr{L}}:=\Big(E\Big[\int_{0}^{T}\int_{Z}\|f(t,z,\omega)\|^{2}\nu(dz)dt\Big]\Big)^{1/2}.$

In the following, when $f(t,z,\omega)$ (or $F(t,z,\omega)$) to appear as the integrand in an integral, for brevity, it will be denoted by $f_{t}(z)$ (or $F_{t}(z)$ respectively).

In an M-type 2 Banach space (Definition \ref{def:1}), by using the usual method, for any $f\in \mathscr{L}$, the integrals

$$
J_{t}(f)=\int_{0}^{t+}\int_{Z}f_{\tau}(z)N(d\tau dz),\ for \ t\geq 0
$$

and

$$
I_{t}(f)=\int_{0}^{t+}\int_{Z}f_{\tau}(z)\tilde{N}(d\tau dz), \ for \ t\geq 0
$$
are well defined. See for e.g.\cite{2012} and references therein.

\begin{definition}(\cite{Brz})\label{def:1}
A Banach space $(\frak X, \|\cdot\|)$ is called M-type 2 if and only
if there exists a constant $C_{\frak X}>0$ such that for any $\frak
X$-valued martingale $\{\bf {M}_{k}\}$, it holds that
\begin{equation}\label{eq:4.1}
 \sup_{k}E[\|{\bf {M}}_{k}\|^{2}]\leq C_{\frak X}\sum_{k}E[\|{\bf
{M}}_{k}-{\bf {M}}_{k-1}\|^{2}].
\end{equation}
\end{definition}

About integral processes $(J_{t})_{t\in [0,T]}$ and $(I_{t})_{t\in [0,T]}$, the following known results will be used to prove some properties of the set-valued case.
\begin{theorem}  \label{thm:000}
 Let $\frak X$ be of M-type 2 and $(Z,\mathcal{B}(Z))$ a
 separable Banach space with finite measure $\nu$,
 ${\bf p}$ a stationary Poisson process with the characteristic measure
 $\nu$. Taking $f\in \mathscr{L}$, then $(J_{t})_{t\in [0,T]}$ and $(I_{t})_{t\in [0,T]}$ are uniformly square integrable, right continuous $\mathcal{F}_{t}$-adapted processes. $(I_{t})_{t\in [0,T]}$ is a martingale with mean zero and $E[J_{t}(f)]=\int_{0}^{t}\int_{Z}E[f_{s}(z)]ds\nu(dz)$.
 Moreover, there exists a constant $C$ such
that
  \begin{equation}
 E\Big[\sup_{0<s\leq t}\Big\|\int_{0}^{s+}\int_{Z}f_{\tau}(z)\tilde{N}(d\tau dz)\Big\|^{2}\Big]
\leq C\int_{0}^{t}\int_{Z}E[\|f_{\tau}(z)\|^{2}]d\tau\nu(dz),
 \end{equation}
 and
 \begin{equation}
 E\Big[\sup_{0<s\leq t}\Big\|\int_{0}^{s+}\int_{Z}f_{\tau}(z)N(d\tau dz)\Big\|^{2}\Big]
\leq C\int_{0}^{t}\int_{Z}E[\|f_{\tau}(z)\|^{2}]d\tau\nu(dz),
 \end{equation}
 where $C$ depends on the constant $C_{\frak X}$ in Definition \ref{def:1}.
 \end{theorem}
 \subsection{Set-valued stochastic integrals w.r.t. Poisson processes}
For the convenience to read the paper without aiding references, and in order to prove our main results, in this subsection, at first we review the stochastic integral of
a set-valued stochastic process with respect to the Poisson point
process and list some auxiliary results obtained in \cite{2012}. And then
we shall study its $L^{2}$-integrable boundedness, submartingale propoerty  and some inequalities,
which make it possible to study the set-valued stochastic differential
equation with set-valued jump.

 A set-valued stochastic process $F=\{F(t)\} :[0,T]\times Z\times\Omega\rightarrow {\bf K}(\frak X)$
is called {\em $\mathscr{S}$-predictable} if $F$ is
$\mathscr{S}/\sigma(\mathcal C)$-measurable.

Set
\begin{equation*}
\mathscr{M}=\Big\{F:[0,T]\times Z\time \Omega\rightarrow {\bf K}(\frak X), F \ is\ \mathscr{S}{\rm-}predictable \ and
\ E\Big[\int_{0}^{T}\int_{Z}\|F_{t}(z)\|_{\bf
K}^{2}dt\nu(dz)\Big]<\infty\Big\}
\end{equation*}
Given a set-valued stochastic process $\{F(t)\}_{t\in[0,T]}$, the
$\frak X$-valued stochastic process $\{f(t)\}_{t\in[0,T]}$ is called
 an {\em $\mathscr{S}$-selection} if $f(t,z,\omega)\in F(t,z,\omega)$
for all $(t,z,\omega)$ and $f\in \mathscr{S}$. For any $F\in\mathscr{M}$, the $\mathscr{S}$-selection
exists and satisfies
$$E\Big[\int_{0}^{T}\int_{Z}\|f_{t}(z)\|^{2}dt\nu(dz)\Big]\leq
E\Big[\int_{0}^{T}\int_{Z}\|F_{t}(z)\|_{\bf
K}^{2}dt\nu(dz)\Big]<\infty,$$ which means $f\in\mathscr{L}$. The
family of all $f$ which belongs to $\mathscr{L}$ and satisfies
$f(t,z,\omega)\in F(t,z,\omega) \ for \ a.e. \ (t,z,\omega)$ is
denoted by $S(F)$, that is
$S(F)=\{f\in \mathscr{L}: f(t,z,\omega)\in F(t,z,\omega) \ for \ a.e. \ (t,z,\omega)\}.$
 Set
\begin{equation*}\tilde{\Gamma} _{t}:=\{\int_{0}^{t+}\int_{Z}f_{s}(z)\tilde{N}(dsdz):
(f (t))_{t\in[0, T]}\in S(F)\},
\end{equation*}
\begin{equation*}\Gamma _{t}:=\{\int_{0}^{t}\int_{Z}f_{s}(z)N(dsdz):
(f (t))_{t\in[0, T]}\in S(F)\}.
\end{equation*}
Let $\overline{de}\tilde{\Gamma}_{t}$
($\overline{de}\Gamma _{t}$) denote the decomposable closed hull of
$\tilde{\Gamma}_{t}$ (resp. $\Gamma _{t}$)with respect to $\mathcal
F_{t}$, where the closure is taken in $L^{1}(\Omega, \frak X)$. Then  $\overline{de}\tilde{\Gamma}_{t}$ and $\overline{de}\Gamma _{t}$ can determine two set-valued random variables respectively, denoted by $I_{t}(F)$, $J_{t}(F)(\in \mathcal M (\Omega, \mathcal
F_{t}, P; {\bf K}(\frak X)))$ such that $S_{I_{t}(F)}^{1}({\mathcal
{F}}_{t})= \overline{de} \tilde{\Gamma}_{t}$ and $S_{J_{t}(F)}^{1}({\mathcal
{F}}_{t})= \overline{de} \Gamma_{t}$.

\begin{definition}\label{def:integral}
The set-valued stochastic processes $(J_{t}(F))_{t\in[0, T]}$ and
$(I_{t}(F))_{t\in[0, T]}$ determined as above are called the stochastic
integrals of $\{F_{t}, : t\in[0, T]\}\in
\mathscr{M}$ with respect to the Poisson random measure $N(ds,dz)$
and the compensated random measure $\tilde{N}(dsdz)$ respectively.
For each $t$, we denote
$I_{t}(F)=\int_{0}^{t+}\int_{Z}F_{s}(z) \tilde{N}(dsdz)$,
$J_{t}(F)=\int_{0}^{t+}\int_{Z}F_{s}(z) N(dsdz)$. Similarly,
for $0< s <t$,we also can define the set-valued random variable
$I_{s,t}(F)=\int_{s}^{t}\int_{Z}F_{\tau}(z)\tilde{N}(d\tau
dz)$, $J_{s,t}(F)=\int_{s}^{t}\int_{Z}F_{\tau}(z)N(d\tau dz)$.
\end{definition}

By additive property of set-valued random variable and Definition \ref{def:integral}, it
is easy to get the proposition below:
\begin{proposition}\label{pro:integral sum}
Assume set-valued stochastic processes $\{F_{t}, \mathcal{F}_{t} :
t\in[0, T]\}$ and $ \{G_{t} , \mathcal{F}_{t} : t\in[0, T]\}\in
\mathscr{M}$. Then
$$J_{t}(F+G)=cl\{J_{t}(F)+J_{t}(G)\} \ a.s \ and \ I_{t}(F+G)=cl\{I_{t}(F)+I_{t}(G)\} \ a.s.,$$
where the {\em cl} stands for the closure in $\frak X$.
\end{proposition}

Assume a set-valued stochastic process $\{F_{t} , \mathcal{F}_{t}:
t\in[0, T]\} \in \mathscr{M}$. From \cite{2012}, we know that $\{J_{t}(F)\}$ and
$\{I_{t}(F)\}$ are integrably bounded and right continuous with respect to $t$. Moreover, if $F_{t}$ is convex for each $t$, the process $\{I_{t}(F), t\in[0,T]\}$ is a set-valued submartingale.

Note:
The integral $\{I_{t}(F), t\in[0,T]\}$ is not a set-valued martingale except for special
case (the singletons). The counterexample and rigorous proof are
given below.

Since the compensated Poisson random measure is a signed measure. In order to give a counterexample, we now decompose it into the difference of two measures. For convenience, let's review the Hahn decomposition of a space and the Jordan decomposition of a signed measure (see e.g.\cite{Meg}).

 The signed measure
$\tilde{N}$ is  defined in the product space $([0,t]\times Z;
\mathcal{B}([0,t])\otimes \mathcal{Z})$ with finite variation. By
the Hahn decomposition theorem, for any fixed $0<t\leq T$, there
exists an essential unique $\mathcal{B}([0,t])\otimes
\mathcal{Z}$-measurable Hahn decomposition denoted by $A^{+}$ and
$A^{-}$ such that $A^{+}\cup A^{-}=(0,t]\times Z$, $A^{+}\cap
A^{-}=\emptyset $, and for any $\mathcal{B}([0,t])\otimes
\mathcal{Z}$-measurable set $B\subset A^{+}$, $\tilde{N}(B)\geq 0$,
for any $\mathcal{B}([0,t])\otimes \mathcal{Z}$-measurable set
$B\subset A^{-}$, $\tilde{N}(B)\leq 0$. The corresponding unique
Jordan decomposition of the signed measure $\tilde{N}$ is denoted by $\tilde{N}^{+}$ and
$\tilde{N}^{-}$ such that $\tilde{N}=\tilde{N}^{+}-\tilde{N}^{-}$.
 For any $\mathcal{B}([0,t])\otimes
\mathcal{Z}$-measurable set $B$,
$$\tilde{N}^{+}(B):=\tilde{N}(B\cap A^{+})=\sup_{S\in B\cap \mathcal{B}([0,t])\otimes
\mathcal{Z}}\tilde{N}(S) $$  and
$$\tilde{N}^{-}(B):=-\tilde{N}(B\cap A^{-})=-\inf_{S\in B\cap \mathcal{B}([0,t])\otimes
\mathcal{Z}}\tilde{N}(S).$$
Particularly,
$$\tilde{N}^{+}(A^{-})=0 \ and\  \tilde{N}^{-}(A^{+})=0.$$
Therefore we have
$$\tilde{N}(B)=\tilde{N}^{+}(B)-\tilde{N}^{-}(B)= \ and\  |\tilde{N}|(B)=\tilde{N}^{+}(B)+\tilde{N}^{-}(B).$$
In addition,  the Jordan decomposition is the minimum decomposition,
and
$$\tilde{N}(dsdz)=N(dsdz)-ds\nu
(dz)=\tilde{N^{+}}(dsdz)-\tilde{N}^{-}(dsdz).$$ Then $\tilde{N}^{+}$
and $\tilde{N}^{-}$ are of finite variation since both $N(dsdz)$ and
$ds\nu(dz)$ are of finite variation. Therefore, for any $f\in
\mathscr{L}$, the integrals
$\int_{0}^{t+}\int_{Z}f_{s}(z)\tilde{N}^{+}(dsdz)$ and
$\int_{0}^{t+}\int_{Z}f_{s}(z)\tilde{N}^{-}(dsdz)$ are well
defined as a manner similar to $\int_{0}^{t+}\int_{Z}f_{s}(z){N}(dsdz)$.
Then we have
$$\int_{0}^{t+}\int_{Z}f_{s}(z)\tilde{N}(dsdz)=\int_{0}^{t+}\int_{Z}f_{s}(z)\tilde{N}^{+}(dsdz)-
\int_{0}^{t+}\int_{Z}f_{s}(z)\tilde{N}^{-}(dsdz),$$

$$\int_{0}^{t+}\int_{Z}f_{s}(z)|\tilde{N}|(dsdz)=\int_{0}^{t+}\int_{Z}f_{s}(z)\tilde{N}^{+}(dsdz)+
\int_{0}^{t+}\int_{Z}f_{s}(z)\tilde{N}^{-}(dsdz).$$

Now we give an example to show that the interval-valued stochastic
integral with respect to the compensated Poisson random measure is
not an interval-valued martingale.

\begin{example}
Let $\frak X=\mathbb{R}$. Take a set-valued stochastic process
$$F(t,z,\omega)\equiv [-1,1] \ for
\ all \ (t, z, \omega)\in [0,T]\times Z\times \Omega. $$
 Then $F=\{F(t,z,\omega) , \mathcal{F}_{t}: t\in[0,T]\}\in \mathscr{M}$.
 The
integral $\int_{0}^{t+}\int_{Z}F_{s}(z)\tilde{N}(dsdz)$ is a
closed interval since $F$ is an interval.  For any selection $f\in
\mathscr{L}$, we have $-1\leq f(s,z,\omega)\leq 1$. Then for any
fixed $t$ ($0<t\leq T$),
\begin{equation*}
\begin{split}
&|\int_{0}^{t+}\int_{Z}f_{s}(z)\tilde{N}(dsdz)|\leq
\int_{0}^{t+}\int_{Z}|f_{s}(z)||\tilde{N}|(dsdz)|\\
&\leq \int_{0}^{t+}\int_{Z}|\tilde{N}|(dsdz)|=|\tilde{N}|((0,t], Z).
\end{split}
\end{equation*}
The extreme point can be attained. In fact, let  $A^{+}$ and $A^{-}$ be the Hahn decomposition
of $(0,t]\times Z$. Taking
$$h(s,z,\omega)=\chi_{A^{+}}-\chi_{A^{-}},$$
which  is non-random then $\mathcal{B}([0,t])\otimes
\mathcal{B}(Z)\otimes \mathcal {F}_{0}$-measurable. Furthermore,
$h=\{h_{t}, \mathcal{F}_{t}: t\in[0,T]\}\in \mathscr{L}$. Then
\begin{equation}\label{eq:014}
\begin{split}
&\sup_{\ \ all \ selections \ f\in
\mathscr{L}}\int_{0}^{t+}\int_{Z}f_{s}(z)\tilde{N}(dsdz)=\int_{0}^{t+}\int_{Z}h_{s}(z)\tilde{N}(dsdz)\\
&=\int_{A^{+}}\tilde{N}(dsdz)-\int_{A^{-}}\tilde{N}(dsdz)\\
&=\tilde{N}(A^{+})-\tilde{N}(A^{-})=\tilde{N}^{+}(A^{+})+\tilde{N}^{-}(A^{-})=|\tilde{N}|((0,t],
Z)
\end{split}
\end{equation}
Similarly, taking $$h(s,z,\omega)=-\chi_{A^{+}}+\chi_{A^{-}},$$ we
obtain
\begin{equation}\label{eq:015}
\inf_{\ \ all \ selections \ f\in
\mathscr{L}}\int_{0}^{t+}\int_{Z}f_{s}(z)\tilde{N}(dsdz)=-|\tilde{N}|((0,t],
Z).
\end{equation}
By the convexity and closedness of
$\int_{0}^{t+}\int_{Z}F_{s}(z)\tilde{N}(dsdz)$, together with \eqref{eq:014}
and \eqref{eq:015}, we obtain
$$\int_{0}^{t+}\int_{Z}F_{s}(z)\tilde{N}(dsdz)=[-|\tilde{N}|((0,t],
Z), |\tilde{N}|((0,t], Z)].$$ It is obvious that the left end point and the right
end point are
$\mathcal{F}_{t}$-supermartingale and submartingale respectively. But not $\mathcal{F}_{t}$-martingale except for
$|\tilde{N}|((0,t],Z)\equiv 0$, a contradiction. Therefore, by
 Definition \ref{def:001} and
Theorem \ref{pro:003}, the integral process
$\{\int_{0}^{t+}\int_{Z}F_{s}(z)\tilde{N}(dsdz),
\mathcal{F}_{t}: t\in(0,T]\}$ is an interval-valued
submartingale, but not an interval-valued martingale.
\end{example}

An interval $I$ is called {\em proper} if it has infinitely many
elements. A convex set $A$ is called {\em non-degenerate} if it has
infinitely many elements. A set $A$ is called a {\em singleton} if
it has only one element. In the following, we will show that for any
 interval-valued integrable stochastic process
$F=\{F(t,z,\omega)\}=\{[f(t,z,\omega), g(t,z,\omega)] ,
\mathcal{F}_{t}: t\in[0,T]\}$, the integral process $\{I_{t}(F) ,
\mathcal{F}_{t}:t\in (0,T]\}$ is not an
interval-valued martingale unless the interval process
$F$ degenerates into single-valued one.

\begin{theorem}\label{thm:interval}
Assume a proper interval-valued stochastic process $\{F_{t}=[f_{t},
g_{t}] , \mathcal{F}_{t}: t\in [0,T]\}$ is integrable with respect
to $N(dsdz)$ and $\tilde{N}(dsdz)$. Then the integral $\{I_{t}(F) ,
\mathcal{F}_{t}:t\in (0,T]\}$ is not an interval-valued martingale.
\end{theorem}
\begin{proof}
Let $\frak X=\mathbb{R}$.

 {\em Step 1}:  At first we consider the
symmetric proper interval. Assume $f=\{f_{t} , \mathcal{F}_{t}: t\in
[0,T]\}\in \mathscr{L}$ and for each $t$, $f_{t}>0$ for a.e.
$(z,\omega)$. Then the interval stochastic process $\{F_{t}=[-f_{t},
f_{t}] , \mathcal{F}_{t}: t\in[0,T]\}\in \mathscr{M}$. Taking any
selection $h\in \mathscr{L}$, we have
\begin{equation*}
\begin{split}
&|\int_{0}^{t+}\int_{Z}h_{s}(z)\tilde{N}(dsdz)|\leq
\int_{0}^{t+}\int_{Z}|h_{s}(z)||\tilde{N}|(dsdz)\\
&\leq \int_{0}^{t+}\int_{Z}f_{s}(z)|\tilde{N}|(dsdz).
\end{split}
\end{equation*}
Similarly, by taking
$$h^{1}(s,z,\omega)=\chi_{A^{+}}f(s,z,\omega)-\chi_{A^{-}}f(s,z,\omega),$$
and respectively
$$h^{2}(s,z,\omega)=-\chi_{A^{+}}f(s,z,\omega)+\chi_{A^{-}}f(s,z,\omega),$$
the extreme points
$\int_{0}^{t+}\int_{Z}f_{s}(z)|\tilde{N}|(dsdz)$ and
$-\int_{0}^{t+}\int_{Z}f_{s}(z)|\tilde{N}|(dsdz)$ can be
attained respectively. Therefore, by the closedness and convexity of the
integral, we obtain
$$\int_{0}^{t+}\int_{Z}F_{s}\tilde{N}(dsdz)=[-\int_{0}^{t+}\int_{Z}f_{s}(z)|\tilde{N}|(dsdz),\int_{0}^{t+}\int_{Z}f_{s}(z)|\tilde{N}|(dsdz)],$$
which implies for each $t$, the integral is a proper interval for
a.e. $(z,\omega)$. The $\mathbb{R}$-valued stochastic process
$\{\int_{0}^{t+}\int_{Z}f_{s}(z)|\tilde{N}|(dsdz),
\mathcal{F}_{t}: t\in(0,T]\}$ is a submartingale but not a
martingale. Indeed, for any $0<s<t\leq T$,
\begin{equation*}
\begin{split}
&E\Big[
\int_{0}^{t+}\int_{Z}f_{\tau}(z)|\tilde{N}|(d\tau dz)|\mathcal{F}_{s}\Big]\\
&=E\Big[
\int_{0}^{s+}\int_{Z}f_{\tau}(z)|\tilde{N}|(d\tau dz)|\mathcal{F}_{s}\Big]+E\Big[
\int_{s}^{t+}\int_{Z}f_{\tau}(z)|\tilde{N}|(d\tau dz)|\mathcal{F}_{s}\Big]\\
&=\int_{0}^{s+}\int_{Z}f_{\tau}(z)|\tilde{N}|(d\tau dz)+E\Big[
\int_{s}^{t+}\int_{Z}f_{\tau}(z)|\tilde{N}|(d\tau dz)|\mathcal{F}_{s}\Big]\\
&\geq\int_{0}^{s+}\int_{Z}f_{\tau}(z)|\tilde{N}|(d\tau dz)
\end{split}
\end{equation*}
but the equality does not always hold for all $0<s<t$ since $$E\Big[
\int_{s}^{t+}\int_{Z}f_{\tau}(z)|\tilde{N}|(d\tau dz)|\mathcal{F}_{s}\Big]=0$$
does not always hold for all $s$.

{\em Step 2}. Let $0<f=\{f_{t}, \mathcal{F}_{t}: t\in[0,T]\}\in
\mathscr{L}$. Setting $F_{t}=[0,f_{t}]$, $G^{1}_{t}=[-f_{t}, f_{t}]
$and $G^{2}_{t}=\{f_{t}\}$ for all $t\in [0,T]$. Then
$2F_{t}=G^{1}_{t}+G^{2}_{t}$.  $F=\{F_{t}, \mathcal{F}_{t}:
t\in[0,T]\}$, $G^{1}=\{G^{1}_{t}, \mathcal{F}_{t}: t\in[0,T]\}$ and
$G^{2}=\{G^{2}_{t}, \mathcal{F}_{t}: t\in[0,T]\}$ belong to
$\mathscr{M}$. By Proposition \ref{pro:integral sum}, for every
$t\in(0,T]$
$$I_{t}(2F)=2I_{t}(F)=\{I_{t}(G^{1})+I_{t}(G^{2})\} \ a.s.$$
Note: here we need not to take closure since bounded closed set is
compact in $\mathbb{R}$, then the set of sum  is closed.

The integral process $\{I_{t}({G^{2}}), \mathcal{F}_{t}:
t\in[0,T]\}$ is an $\mathbb{R}$-valued martingale. And the process
$\{I_{t}({G^{1}}), \mathcal{F}_{t}: t\in[0,T]\}$ is not an
interval-valued martingale but an interval-valued submartingale. Then
the sum $\{I_{t}(2F),  \mathcal{F}_{t}: t\in[0,T]\}$ is an
interval-valued submartingale but not an interval-valued martingale, so does
$\{I_{t}(F),  \mathcal{F}_{t}: t\in[0,T]\}$.

{\em Step 3.} Assume $f=\{f_{t}, \mathcal{F}_{t}: t\in[0,T]\}$,
$g=\{g_{t}, \mathcal{F}_{t}: t\in[0,T]\}\in \mathscr{L}$ and
$f(t,z,\omega)<g(t,z,\omega)$. Setting $F_{t}=[f_{t}, g_{t}]$ for
all $t$, then $F=\{F_{t},  \mathcal{F}_{t}: t\in[0,T]\}\in
\mathscr{M}$ and $F_{t}=\{f_{t}\}+[0,g_{t}-f_{t}]$. In a similar way
as the proof of Step 2, we obtain $\{I_{t}(F),  \mathcal{F}_{t}:
t\in[0,T]\}$ is an interval-valued submartingale but not an
interval-valued martingale.

From the above proof, we obtain that the integral process
$\{I_{t}(F), \mathcal{F}_{t}: t\in[0,T]\}$ is a martingale if and
only if the integrand degenerates into a real valued process.
\end{proof}

In order to prove the result being also true for M-type 2 Banach space
$\frak X$, we aid the bounded linear functional $x^{*}$, which is defined on $\frak X$ and takes values in $\mathbb R$. Let
$\frak X^{*}$ be the family of all bounded linear functionals, i.e. the dual space of $\frak X$, $F=
\{F_{t}, {\mathcal{F}}_{t}: t\in[0, T]\}$ be a convex set-valued
stochastic process. Taking $x^{*}\in\frak X^{*}$, for any $t\in [0,
T]$, define
\begin{equation}\label{eq:016}
F_{t}^{x^{*}}(\omega):=cl\{<x^{*}, a>: a\in F_{t}(\omega)\} \ for \
\omega\in\Omega,
\end{equation}
then $F_{t}^{x^{*}}$ is an interval-valued $\mathcal {F}_{t}$-
measurable random variable (Note: for some $t$, the interval
$F_{t}^{x^{*}}$ may be a singleton for a.e. $(z,\omega)$. For
instance, the case $x^{*}$=0 ). Indeed, it is convex since the
convexity of $F_{t}(\omega)$ and the linearity of $x^{*}$. Take any
open interval $(c, d)\subset \mathbb{R}$,
\begin{eqnarray*}
& &\{\omega: F_{t}^{x^{*}}(\omega)\cap(c, d)\neq \emptyset
\}\\
&=& \Omega \setminus(\{\omega: \sup_{a\in F_{t}(\omega)} <x^{*},
a>\leq c\}\cup \{\omega: \inf_{a\in F_{t}(\omega)}<x^{*}, a>\geq
d\}) \in \mathcal{F}_{t},
\end{eqnarray*}
i.e. $F_{t}^{x^{*}}$ is $\mathcal{F}_{t}$-measurable. Further,
\begin{equation}\label{eq:017}
S_{F_{t}^{x^{*}}}^{1}(\mathcal{F}_{t})=cl\{<x^{*}, f_{t}>: f_{t} \in
S_{F_{t}}^{1}(\mathcal{F}_{t})\}.
\end{equation}
Therefore, $\{F_{t}^{x^{*}}: t\in [0, T]\}$ is an interval-valued
stochastic process. Moreover, if $F=\{F_{t}, {\mathcal{F}}_ {t}:
t\in[0, T]\}$ is convex and belongs to $\mathscr{M}$, then
$F^{x^{*}}=\{F^{x^{*}}_{t}, {\mathcal{F}}_{t}: t\in[0, T]\}$ is an
integrable interval and
$$
S(F^{x^{*}})=cl \{<x^{*}, f>: f=(f_{s})_{s\in [0, T]}\in S(F)\},
$$
where the closure is taken in product space $L^{2}(([0, T] \times
Z\times \Omega), \mathcal {B}([0, T])\otimes \mathcal{B}(Z)\otimes
\mathcal {F}, \lambda\times \nu\times P; \frak X)$, $\lambda$ the
Lebesgue measure in $([0,T];\mathcal{B}([0,T])).$

In a manner similar to the proof of Theorem 4.5 in \cite{Zha4}, we
have the following theorems:
\begin{theorem}\label{th:weak integral}
Assume convex set-valued stochastic process $\{F_{t}, {\mathcal{F}}_
{t}: t\in[0, T]\}\in {\mathscr {M}}$. Then for any $x^{*}\in \frak
X^{*}$, $\{I^{x^{*}}_{t}(F), {\mathcal{F}}_{t}: t \in[0, T]\}$ and
$\{J^{x^{*}}_{t}(F), {\mathcal{F}}_{t}: t \in[0, T]\}$ are
interval-valued $\mathcal{F}_{t}$-adapted processes. For any
$t\in[0, T]$,
\begin{equation*}
\begin{split}
&I_{t}^{x^{*}}(F)(\omega)=I_{t}(F^{x^{*}})(\omega) a.s.\\
&J_{t}^{x^{*}}(F)(\omega)=J_{t}(F^{x^{*}})(\omega) a.s.
\end{split}
\end{equation*}
where $I_{t}^{x^{*}}(F)(\omega):=(I_{t}(F))^{x^{*}} (\omega)$ and
$J_{t}^{x^{*}}(F)(\omega):=(J_{t}(F))^{x^{*}} (\omega)$.
\end{theorem}
\begin{theorem}\label{th:002}
Assume convex set-valued stochastic process $\{F_{t}, {\mathcal{F}}_
{t}: t\in[0, T]\}\in \mathscr{M}$, then for any $x^{*}\in \frak
X^{*}$ and $s<t\in[0, T]$,
\begin{equation*}
\begin{split}
&E[I_{t}^{x^{*}}(F)|\mathcal {F}_{s}](\omega)=E^{x^{*}}[I_{t}(F)
|\mathcal {F}_{s}](\omega)\ a.s.\\
&E[J_{t}^{x^{*}}(F)|\mathcal {F}_{s}](\omega)=E^{x^{*}}[J_{t}(F)
|\mathcal {F}_{s}](\omega)\ a.s.
\end{split}
\end{equation*}
where $E^{x^{*}}[I_{t}(F) |\mathcal {F}_{s}]$ is the $\bf {K}(\frak
X)$-valued random variable determined by
\begin{equation*}
\begin{split}
&S_{E^{x^{*}}[I_{t}(F)|\mathcal{F}_{s}]}^{1}(\mathcal{F}_{s})\\
&=cl\{<x^ {*}, g>: g\in
S_{E[I_{t}(F)|\mathcal{F}_{s}]}^{1}(\mathcal{F}_{s})\}
 =cl\{<x^{*},E[g_{t}|\mathcal{F}_{s}]>: g_{t}\in S_{I_{t}(F)}^{1} (\mathcal{F}_{t})\}
 \end{split}
 \end{equation*}
 and $E^{x^{*}}[J_{t}(F) |\mathcal {F}_{s}]$ is the $\bf {K}(\frak
X)$-valued random variable determined by
\begin{equation*}
\begin{split}
&S_{E^{x^{*}}[J_{t}(F)|\mathcal{F}_{s}]}^{1}(\mathcal{F}_{s})\\
&=cl\{<x^ {*}, g>: g\in
S_{E[J_{t}(F)|\mathcal{F}_{s}]}^{1}(\mathcal{F}_{s})\}
 =cl\{<x^{*},E[g_{t}|\mathcal{F}_{s}]>: g_{t}\in S_{J_{t}(F)}^{1}
 (\mathcal{F}_{t})\}.
 \end{split}
 \end{equation*}
\end{theorem}
\begin{theorem}\label{thm:unnmartingale}
Assume a non-degenerate convex set-valued stochastic process
$F=\{F_{t}: t\in[0, T]\} \in \mathscr{M}$. The integral process
$\{I_{t}(F): t\in (0,T]\}$ is not a set-valued martingale
\end{theorem}
\begin{proof}
Since  for each $t$, $F_{t}$ is a non-degenerate closed convex
subset of $\frak X$ for a.e. $(z,\omega)$ then so is the integral
$I_{T}(F)$ a.s. Moreover, the expectation $E[I_{T}(F)]$ is a
non-degenerate convex subset of $\frak X$. Taking $x,y\in
E[I_{T}(F)]$ and $x\neq y$, by using the Hahn-Banach extension
theorem of functionals,  there exists an $x^{*}\in \frak X^{*}$ (
independent of $t,z,\omega$), such that $x^{*}(x)\neq x^{*}(y)$. Indeed, set $A(x,y)=\{a(x-y): a\in\mathbb{R}\}$ and define $<\phi,
a(x-y)>=a\|x-y\|$ for every $a(x-y)\in A(x,y)$. It is observed that
$\phi$ is a bounded linear functional with operator norm
$\|\phi\|=1$ defined in the subspace $A(x,y)$. Then by the
Hahn-Banach extension theorem (c.f.\cite{Meg}), there exists a
bounded linear functional $x^{*}: \frak X\rightarrow \mathbb{R}$
such that $x^{*}$ restricted to $A(x,y)$ is equal to $\phi$ and
$\|x^{*}\|=1$.

\noindent On the other hand, for the linearity of $x^{*}$,
 \begin{equation*}
<x^{*}, E\Big[\int_{0}^{T+}\int_{Z}F_{s}(z)\tilde{N}(dsdz)\Big]>
= E\Big[\int_{0}^{T+}\int_{Z}<x^{*},
F_{s}(z)>\tilde{N}(dsdz)\Big].
 \end{equation*}
 Since
 $E\Big[\int_{0}^{T+}\int_{Z}F(s,z,\omega)\tilde{N}(dsdz)\Big]$ is
 a non-degenerate convex subset of $\frak X$, by the choice of
 $x^{*}$, $<x^{*},
 E\Big[\int_{0}^{T+}\int_{Z}F_{s}(z)\tilde{N}(dsdz)\Big]>$ is
 a proper interval, which implies $E\Big[\int_{0}^{T+}\int_{Z}<x^{*},
F(s,z,\omega)>\tilde{N}(dsdz)\Big]$ is a proper interval.
Furthermore, by the convexity of $F(s,z,\omega)$, $<x^{*},
F(s,z,\omega)>$ is a proper interval for a.e.$(s,z,\omega)$. That
means the stochastic processes $F^{x^{*}}$ is a proper
interval-valued process.

By Theorem \ref{thm:interval}, the interval-valued stochastic
process $\{I_{t}(F^{x^{*}}), \mathcal{F}_{t}:t\in(0,T]\}$ is not a
proper interval-valued martingale. As a further result, we will show
that $\{I_{t}(F), \mathcal{F}_{t}:t\in(0,T]\}$ is not a ${\bf
K}(\frak X)$-valued martingale.

Otherwise, suppose $\{I_{t}(F), \mathcal{F}_{t} :t\in(0,T]\}$ is  a
${\bf K}(\frak X)$-valued martingale. Taking the same functional
$x^{*}$ as above, for each $t$, both $I_{t}(F^{x^{*}})$ and
$I_{t}(F)$ are non-degenerate convex sets a.s. Then
$\{I_{t}(F^{x^{*}}):t\in (0,T]\}$ is a proper interval stochastic
process. By the property of set-valued martingale, we have
\begin{equation}\label{eq:013}
S_{I_{s}(F)}^{1}(\mathcal{F}_{s})=S_{E[I_{t}(F)|\mathcal{F}_{s}]}^{1}(\mathcal{F}_{s}).
\end{equation}
According to Theorem \ref{th:weak integral}, for any $0<s<t\leq T$,
we obtain
\begin{equation*}
\begin{split}
S_{I_{s}(F^{x^{*}})}^{1}(\mathcal{F}_{s})&=cl\{<x^{*}, g_{s}>:
g_{s}\in S_{I_{s}(F)}^{1}(\mathcal{F}_{s})\}\\
&=cl\{<x^{*}, g_{s}>:
g_{s}\in S_{E[I_{t}(F)|\mathcal{F}_{s}]}^{1}(\mathcal{F}_{s})\} \ (by\ \eqref{eq:013})\\
&=S_{E^{x^{*}}[I_{t}(F)|\mathcal{F}_{s}]}^{1}(\mathcal{F}_{s}) \ (\
by\ \eqref{eq:017})\\
&=S_{E[I^{x^{*}}_{t}(F)|\mathcal{F}_{s}]}^{1}(\mathcal{F}_{s})\ (by\
Theorem\ \ref{th:002})\\
&=S_{E[I_{t}(F^{x^{*}})|\mathcal{F}_{s}]}^{1}(\mathcal{F}_{s})\ (by\
Theorem\ \ref{th:weak integral}),
\end{split}
\end{equation*}
which implies that $\{I_{t}(F^{x^{*}}), \mathcal{F}_{t}:t\in(0,T]\}$
is a proper interval-valued martingale according to the equivalent
conditions (see e.g. Theorem 3.1. in \cite{Zha4}), a contradiction to
Theorem \ref{thm:interval}.
\end{proof}
\begin{remark}
Theorem 3.7 in our previous paper \cite{2012} states that the integral $\{I_{t}(F),t=0,...,T\}$ is a set-valued martingale with a very simple proof. In fact at that time we carelessly misused the martingale equivalent condition of Theorem 3.1. in \cite{Zha4}. The martingale equivalent condition is for each $t$,
$$
 S_{I_{t}(F)}^{1}({\mathcal{F}}_{t})
=cl\{g_{t}: (g_{s})_{s\in [0, T]}\in {\bf
 MS(F)}\},
$$  where $g$ is an $\frak X$-valued martingale. The condition is different from the following Castaing representation
$$
I_{t}(F)(\omega)=cl\{\int_{0}^{t+}\int_{Z}f^{i}_{s}(z)\tilde{N}(dsdz)(\omega):
i=1, 2,...\} \ a.s.
$$
The latter is weaker. So we can not get the martingale property of $\{I_{t}(F), t\in[0, T]\}$ from the Castaing representation even though for each $i$, $\{\int_{0}^{t+}\int_{Z}f^{i}_{s}(z)\tilde{N}(dsdz),t\in [0,T]\}$ is an-$\frak X$-valued martingale.
\end{remark}
In \cite{2012}, Theorem 3.3 shows that $\{I_{t}(F)\}$ and  $\{J_{t}(F)\}$ are $L^{1}$-integrably bounded. We now show the $L^{2}$-integrable boundedness.
Set
$$S_{I_{t}(F)}^{2}(\mathcal F_{t}):=\overline{de}_{L^{2}}\{\int_
{0}^{t+}\int_{Z}f_{s}(z)\tilde{N}(dsdz): (f_{s})_{s\in[0, T]}\in
S(F)\},$$
$$S_{J_{t}(F)}^{2}(\mathcal F_{t}):=\overline{de}_{L^{2}}\{\int_
{0}^{t+}\int_{Z}f_{s}(z)N(dsdz): (f_{s})_{s\in[0, T]}\in S(F)\},$$
where the closure is taken in $L^{2}$. We have the following result:
\begin{lemma}\label{lem:4}
Assume a set-valued stochastic process $\{F_{t}: t\in[0, T]\} \in
\mathscr{M}$. Then for every $t\in[0,T]$ $S_{I_{t}(F)}^{1}(\mathcal
{F}_{t})=S_{I_{t}(F)}^{2}(\mathcal {F}_{t})$, and
$S_{J_{t}(F)}^{1}(\mathcal {F}_{t})=S_{J_{t}(F)}^{2}(\mathcal
{F}_{t})$.
\end{lemma}
\begin{proof}
Obviously, $S_{I_{t}(F)}^{1}(\mathcal {F}_{t})\supset
S_{I_{t}(F)}^{2}(\mathcal {F}_{t})$, and $S_{J_{t}(F)}^{1}(\mathcal
{F}_{t})\supset S_{J_{t}(F)}^{2}(\mathcal {F}_{t})$. It is
sufficient to prove the converse inclusions.

{\em Step 1}: At first we shall show that $de \Gamma_{t}$ and $de
\tilde{\Gamma}_{t}$ are bounded in $L^{2}(\Omega; \frak X)$.

For any finite $\mathcal F_{t}$-measurable partition
$\{A_{1},...,A_{m}\} $ of $\Omega$ and a finite sequence
$\{f^{1},...,f^{m}\} \subset S(F),$
\begin{equation*}
\begin{split}
&E[\|\sum_{i=1}^{m}\chi_{A_{i}}\int_{0}^{t+}\int_{Z}f^{i}_{s}(z)N(dsdz)\|^{2}]\\
&=\sum_{i=1}^{m}E[\chi_{A_{i}}\|\int_{0}^{t+}\int_{Z}f^{i}_{s}(z)N(dsdz)\|^{2}]\\
&\leq\sum_{i=1}^{m}E[\chi_{A_{i}}(\int_{0}^{t+}\int_{Z}\|f^{i}_{s}(z)\|N(dsdz))^{2}]\\
&\leq
\sum_{i=1}^{m}E[\chi_{A_{i}}(\int_{0}^{t+}\int_{Z}\|F_{s}(z)\|_{\bf
K}N(dsdz))^{2}]\\
&=E[(\int_{0}^{t+}\int_{Z}\|F_{s}(z)\|_{\bf K}N(dsdz))^{2}].
\end{split}
\end{equation*}
The process $\{\|F(t)\|_{\bf k}: t\in(0,T]\}$ is a real valued
predictable (with three parameters $t,z,\omega$) process since the
set-valued stochastic process $\{F(t),t\in(0,T]\}$ is
$\mathscr{S}$-predictable.  Then by Theorem \ref{thm:000}, we have
\begin{equation}\label{squarebounded}
E[(\int_{0}^{t+}\int_{Z}\|F_{s}(z)\|_{\bf K}N(dsdz))^{2}]\leq
CE[(\int_{0}^{t}\int_{Z}\|F_{s}(z)\|_{\bf K}^{2}ds\nu(dz))]<\infty,
\end{equation}
where $C$ is the constant that depends on $C_{\frak X}$.
The inequality \eqref{squarebounded} implies that $de\Gamma_{t}$ is bounded in
$L^{2}(\Omega, \frak X)$.

\begin{equation*}
\begin{split}
&E[\|\sum_{i=1}^{m}\chi_{A_{i}}\int_{0}^{t+}\int_{Z}f^{i}_{s}(z)\tilde{N}(dsdz)\|^{2}]\\
&=\sum_{i=1}^{m}E[\chi_{A_{i}}\|\int_{0}^{t+}\int_{Z}f^{i}_{s}(z)\tilde{N}(dsdz)\|^{2}]\\
&\leq
2\sum_{i=1}^{m}E[\chi_{A_{i}}\|\int_{0}^{t+}\int_{Z}f^{i}_{s}(z)N(dsdz)\|^{2}]+
2\sum_{i=1}^{m}E[\chi_{A_{i}}\|\int_{0}^{t}\int_{Z}f^{i}_{s}(z)ds\nu(dz)\|^{2}]\\
&\leq
2\sum_{i=1}^{m}E[\chi_{A_{i}}(\int_{0}^{t+}\int_{Z}\|f^{i}_{s}(z)\|N(dsdz))^{2}]+
2\sum_{i=1}^{m}E[\chi_{A_{i}}(\int_{0}^{t}\int_{Z}\|f^{i}_{s}(z)\|ds\nu(dz))^{2}]\\
 &\leq
2 \sum_{i=1}^{m}E[\chi_{A_{i}}(\int_{0}^{t+}\int_{Z}\|F_{s}(z)\|_{\bf
K}N(dsdz))^{2}]+
2 \sum_{i=1}^{m}E[\chi_{A_{i}}(\int_{0}^{t}\int_{Z}\|F_{s}(z)\|_{\bf K}ds\nu(dz))^{2}]\\
&=2E[(\int_{0}^{t+}\int_{Z}\|F_{s}(z)\|_{\bf K}N(dsdz))^{2}]+
2E[(\int_{0}^{t}\int_{Z}\|F_{s}(z)\|_{\bf K}ds\nu(dz))^{2}]\\
&\leq 2CE[\int_{0}^{t}\int_{Z}\|F_{s}(z)\|_{\bf K}^{2}ds\nu(dz))]+
 2T\nu(Z)E[\int_{0}^{t}\int_{Z}\|F_{s}(z)\|_{\bf K}^{2}ds\nu(dz))]\\
 &=2(C+T\nu(Z))E[\int_{0}^{t}\int_{Z}\|F_{s}(z)\|_{\bf
 K}^{2}ds\nu(dz))]<\infty,
\end{split}
\end{equation*}
which yields that $de\tilde{\Gamma}_{t}$ is bounded in
$L^{2}(\Omega, \frak X)$.

{\em Step 2}. We shall show that the closure of $de \Gamma_{t}$
($de\tilde{\Gamma}_{t}$) in $L^{1}$ is also a subset of $L^{2}(\Omega;
\frak X)$.

Taking any $h\in S_{J_{t}(F)}^{1}(\mathcal {F}_{t})$, there exists a
sequence $$\{h^{k}: k=1,2,...\}\subset de\Gamma_{t},$$ such that
$$E\|h^{k}-h\|\rightarrow 0\ \ as\ k\rightarrow +\infty.$$ Then
there exists a subsequence $\{h^{k_{i}}: i=1,2,...\}$ of $\{h^{k}:
k=1,2,...\}$ such that

$$\|h^{k_{i}}-h\|\rightarrow 0\ as \ i\rightarrow +\infty\ a.s.$$

For any $h^{k}$, we have $h^{k}\leq
\int_{0}^{t+}\int_{Z}\|F_{s}\|_{\bf K}N(dsdz)$ a.s. In addition,
$\int_{0}^{t+}\int_{Z}\|F_{s}\|_{\bf K}N(dsdz)$ is $L^{2}$-integrable.
Therefore, by the Lebesgue dominated convergence theorem, we obtain
$$E[\|h^{k_{i}}-h\|^{2}]\rightarrow 0\ as \ i\rightarrow +\infty\  $$

By the inequality,
$$\|h(\omega)\|^{2}\leq 2\|h(\omega)-h^{k_{i}}(\omega)\|^{2}+2\|h^{k_{i}}(\omega)\|^{2},\ a.s.$$
immediately, we obtain $h\in S^{2}_{J_{t}(F)}(\mathcal{F}_{t})$,
which implies $S^{1}_{J_{t}(F)}(\mathcal{F}_{t})\subset
S^{2}_{J_{t}(F)}(\mathcal{F}_{t})$. Similarly, we have
$S^{1}_{I_{t}(F)}(\mathcal{F}_{t})\subset
S^{2}_{I_{t}(F)}(\mathcal{F}_{t})$.
\end{proof}

By  Lemma \ref{lem:4} and its proof, we
get Theorem \ref{th:squarebounded} below, which is necessary to
guarantee the availability to study the set-valued stochastic
differential equation with set-valued jump part.
\begin{theorem}\label{th:squarebounded}
Assume a set-valued stochastic process $\{F_{t}, {\mathcal{F}}_{t}:
t\in[0, T]\} \in \mathscr{M}$. Then both $\{J_{t}(F)\}$ and
$\{I_{t}(F)\}$ are $L^{2}$-integrably bounded.
\end{theorem}


If $\mathcal{F}$ is separable, by Theorem 3.5 in \cite{2012} , for
stochastic processes $\{I_{t}, \mathcal{F}_{t}: t\in(0,T]\}$ and
$\{J_{t}, \mathcal{F}_{t}: t\in(0,T]\}$, there exist
$\mathcal{F}\otimes\mathcal{B}([0,T])$-measurable and
$\mathcal{F}_{t}$-adapted versions. From now on, we always take the
measurable versions.
\begin{lemma}\label{lem:inequality1}
Assume $\mathcal F$ is separable with respect to $P$. For set-valued
stochastic processes $\{F_{t}\}_{t\in[0, T]}, \{G_{t}\}_{t\in[0,
T]}\in\mathscr{M}$, and for all $t$, we have
\begin{equation}
\begin{split}\label{eqn:110}
&H\Big(\int_{0}^{t+}\int_{Z}F_{s}(z)N(dsdz),\int_{0}^{t+}\int_{Z}G(s,z, \omega)N(dsdz)\Big)\\
&\leq \int_{0}^{t+}\int_{Z}H(F_{s}(z),G_{s}(z))N(dsdz)
\ a.s.
\end{split}
\end{equation}
\end{lemma}
\begin{proof}
By Theorem 3.5 in \cite{2012}, there exists a sequence $\{f^{i}:
i\in\mathbb{N}\}\subset S(F)$, such that
$$
F(t,z, \omega)= \mbox{cl}\big\{f^{i}(t,z, \omega): i\in
\mathbb{N}\big\} \ \mbox{a.e.} \ (t, z, \omega)
$$
and, for each $t\in[0, T]$,
$$
\int_{0}^{t+}\int_{Z}F_{s}(z)N(dsdz)=
\mbox{cl}\Big\{\int_{0}^{t+}\int_{Z}f^{i}_{s}(z)N(dsdz): i\in
\mathbb{N}\Big\}.
$$
For each $i\geq 1$, we can choose a sequence
 $\{g^{ij}: j\in \mathbb{N}\}\subset S(G)$ (this sequence depends on $i$), such that
 $$\|f^{i}-g^{ij}\|_{\mathscr{L}^{1}}\downarrow
 d\left(f^{i}, S(G)\right)
 \ (j\rightarrow +\infty),$$
 where
 $$\|f^{i}-g^{ij}\|_{\mathscr{L}^{1}}=\int_{\Omega}\int_{0}^{T}\int_{Z}\|f^{i}_{s}(z)-g^{ij}_{s}(z)\|ds\nu(dz)dp,$$
 and
$$d\left(f^{i}, S(G)\right)=\inf_{g\in S(G)}\|f^{i}-g\|_{\mathscr{L}^{1}}.$$
In fact,
\begin{equation}\label{eqn:equality}
\begin{split}
&\int_{\Omega}\int_{0}^{T+}\int_{Z}\|f^{i}_{s}(z)-g^{ij}_{s}(z)\|N(dsdz)dp\\
&=\int_{\Omega}\int_{0}^{T}\int_{Z}\|f^{i}_{s}(z)-g^{ij}_{s}(z)\|ds
\nu(dz)dp<+\infty
\end{split}
\end{equation}
since $F, G\in \mathscr{M}$.

 By \eqref{eqn:equality} and Theorem 2.2 in \cite{Hia}, we have
\begin{equation*}
\begin{split}
&d\left(f^{i}, S(G)\right)=\inf\limits_{g\in
 S(F)}\|f^{i}-g\|_{\mathscr{L}^{1}}\\
 &=\inf_{g\in
 S(G)}\int_{\Omega}\int_{0}^{T}\int_{Z}\|f^{i}_{s}(z)-g_{s}(z)\|ds\nu(dz)dP\\
 &=\inf_{g\in
 S(G)}\int_{\Omega}\int_{0}^{T+}\int_{Z}\|f^{i}_{s}(z)-g_{s}(z)\|N(dsdz)dP\\
 &=\inf_{g\in
 S(G)}\int_{\Omega}\int_{0}^{T+}\int_{Z}\|f^{i}_{s}(z)-g_{s}(z)\|N(dsdz)dP\\
 &=\int_{\Omega}\int_{0}^{T+}\int_{Z}\inf\limits_{y\in
 G_{s}(z)}\|f^{i}_{s}(z)-y\|N(dsdz)dP\\
 &=\int_{\Omega}\int_{0}^{T+}\int_{Z}d(f^{i}_{s}(z),
 G_{s}(z))N(dsdz)dP.
 \end{split}
 \end{equation*}
Namely, noticing that $\|f^{i}-g^{ij}\|_{\mathscr{L}^{1}}\geq
 d\left(f^{i}, S(G)\right)$ and $\|f^{i}(s,z, \omega)-g^{ij}(s,z, \omega)\|$ $\geq d(f^{i}(s,z,\omega),
 G(s,z, \omega))$ for a.e.\ $(s, z, \omega)$, then
 for any $\varepsilon>0$, there exists a natural number $M$ such
 that
 for any $j\geq M$,
 \begin{eqnarray*}
 \varepsilon &>&\Big|\int_{\Omega}\int_{0}^{T+}\int_{Z}\|f^{i}_{s}(z)-g^{ij}_{s}(z)\|N(dsdz)dP\\
 &&-\int_{\Omega}\int_{0}^{T+}\int_{Z}d(f^{i}_{s}(z),
 G_{s}(z))N(dsdz)dP\Big|\\
 &=&\int_{\Omega}\int_{0}^{T+}\int_{Z}\|f^{i}_{s}(z)-g^{ij}_{s}(z)\|N(dsdz)dP\\
 &&-\int_{\Omega}\int_{0}^{T+}\int_{Z}d(f^{i}_{s}(z),
 G_{s}(z))N(dsdz)dP\\
 &=&\int_{\Omega}\int_{0}^{T+}\int_{Z}\Big(\|f^{i}_{s}(z)-g^{ij}_{s}(z)\|-d(f^{i}_{s}(z),
 G_{s}(z))\Big)N(dsdz)dP\\
 &=&\int_{\Omega}\int_{0}^{T+}\int_{Z}\Big|\|f^{i}_{s}(z)-g^{ij}_{s}(z)\|-d(f^{i}_{s}(z),
 G_{s}(z))\Big|N(dsdz)dP.
 \end{eqnarray*}
 Hence there exists a subsequence of $\{g^{ij}: j\in \mathbb{N}\}$, denoted
 as $\{g^{ij_{k}}: k\in \mathbb{N}\}$ such that
 $$
\|f^{i}(s,z, \omega)-g^{ij_{k}}(s,z, \omega)\|\rightarrow
 d(f^{i}(s,z, \omega), G(s,z, \omega))
 \ (k\rightarrow +\infty) \ \mbox{a.e.} \ (s, z, \omega).
 $$
 \noindent Because $\{F_{t}\}_{t\in[0, T]}$ and $\{G_{t}\}_{t\in[0,
T]}$ are in $\mathscr{M}$, we have
\begin{equation}\label{eqn:3.51}
\int_{\Omega}\int_{0}^{T+}\int_{Z}(\|F_{s}(z)\|_{\bf K}+
    \|G_{s}(z)\|_{\bf K})N(dsdz)dp < \infty,
\end{equation}
which yields
\begin{equation}\label{eq:dominated}
    \int_{0}^{T+}\int_{Z}(\|F_{s}(z)\|_{\bf K}+
    \|G_{s}(z)\|_{\bf K})N(dsdz)<\infty \ \mbox{a.s.},
\end{equation}
Since
 $$\|f^{i}(s,z, \omega)-g^{ij_{k}}(s,z, \omega)\|\leq \|F(s,z, \omega)\|_{\bf K}+\|G(s,z, \omega)\|_{\bf K}
  \ \mbox{for a.e.} (s, ,z, \omega)$$ together with \eqref{eq:dominated},
by the Lebesgue dominated convergence theorem, for all $t$ and almost sure $\omega$, we obtain that
$$\int_{0}^{t+}\int_{Z}\|f^{i}_{s}(z)-g^{ij_{k}}_{s}(z)\|N(dsdz)\rightarrow \int_{0}^{t+}\int_{Z}d(f^{i}_{s}(z), G_{s}(z))N(dsdz)$$
 when $ k\rightarrow +\infty$.
 Therefore, for all $t$ and almost sure $\omega$
 $$\inf\limits_{k}\int_{0}^{t+}\int_{Z}\|f^{i}_{s}(z)-g^{ij_{k}}_{s}(z)\|N(dsdz)
 \leq \int_{0}^{t+}\int_{Z}d(f^{i}_{s}(z), G_{s}(z))N(dsdz).$$
Hence, for all $t$ and almost sure $\omega$, we have
 \begin{eqnarray*}
 &&\sup\limits_{x\in\int_{0}^{t+}\int_{Z}F_{s}(z)ds}d\Big(x,
 \int_{0}^{t+}\int_{Z}G_{s}(z)N(dsdz)\Big)\\
 &&\leq\sup\limits_{i}\inf\limits_{j}\|\int_{0}^{t+}\int_{Z}f^{i}_{s}(z)N(dsdz)-
 \int_{0}^{t+}\int_{Z}g^{ij}_{s}(z)N(dsdz)\|\\
 &&\leq\sup\limits_{i}\inf\limits_{k}\|\int_{0}^{t+}\int_{Z}f^{i}_{s}(z)N(dsdz)
 -\int_{0}^{t+}\int_{Z}g^{ij_{k}}_{s}(z)ds\|\\
 &&\leq \sup\limits_{i}\inf\limits_{k}\int_{0}^{t+}\int_{Z}\|f^{i}_{s}(z)-g^{ij_{k}}_{s}(z)\|N(dsdz)\\
 &&\leq \sup\limits_{i}\int_{0}^{t+}\int_{Z}d(f^{i}_{s}(z),G_{s}(z))N(dsdz)\\
 &&\leq
 \int_{0}^{t+}\int_{Z}\sup\limits_{i}d(f^{i}_{s}(z),G_{s}(z))N(dsdZ).
 \end{eqnarray*}
 Similarly, by Theorem 3.5 in \cite{2012}, there exists a sequence $\{g^{m}: m\in\mathbb{N}\}\subset
S(G)$ such that
$$
G(t,z, \omega)= \mbox{cl}\big\{g^{m}(t,z, \omega): m\in\mathbb{N}
\big\} \ \mbox{a.e.} \ (t, z, \omega)
$$
and, for each $t\in[0, T]$,
$$
\int_{0}^{t+}\int_{Z}G_{s}(z)N(dsdz)=
\mbox{cl}\Big\{\int_{0}^{t+}\int_{Z}g^{m}_{s}(z)N(dsdz): m\in
\mathbb{N}\Big\}.
$$
 In the same way as above, we obtain that for all $t$ and
 almost sure $\omega$,
 \begin{equation*}
 \begin{split}
&\sup\limits_{y\in\int_{0}^{t+}\int_{Z}G_{s}(z)N(dsdz)}d\Big(y,
 \int_{0}^{t+}\int_{Z}F_{s}(z)N(dsdz)\Big)\\
&\leq \int_{0}^{t+}\int_{Z}\sup\limits_{m}d(g^{m}_{s}(z),F_{s}(z))N(dsdz).
\end{split}
\end{equation*}
Therefore, the inequality
 \begin{equation*}
 \begin{split}
&H\Big(\int_{0}^{t+}\int_{Z}F_{s}(z)N(dsdz),\int_{0}^{t+}\int_{Z}G_{s}(z)N(dsdz)\Big)\\
&\leq \int_{0}^{t+}\int_{Z}H(F_{s}(z),G_{s}(z))N(dsdz)
\end{split}
\end{equation*}
holds for all $t$ and almost sure $\omega$.
\end{proof}
\begin{theorem}\label{thm:300}
Assume $\mathcal F$ is separable with respect to $P$.   Let
$\{F_{t}\}_{t\in[0, T]}$ and $\{G_{t}\}_{t\in[0, T]}$ be set-valued
stochastic processes in $\mathscr{M}$. Then for all $t$, it follows
that
\begin{equation}\label{eqn:111}
\begin{split}
&E\big[H\Big(\int_{0}^{t+}\int_{Z}F_{s}(z)N(dsdz),\int_{0}^{t+}\int_{Z}G_{s}(z)N(dsdz)\Big)\big]\\
&\leq
E\big[\int_{0}^{t+}\int_{Z}H(F_{s}(z),G_{s}(z))N(dsdz)\big]\\
&=E\big[\int_{0}^{t}\int_{Z}H(F_{s}(z),G_{s}(z))ds\nu{dz}\big]
\end{split}
\end{equation}
and
\begin{equation}\label{eqn:112}
\begin{split}
&E\big[H^{2}\Big(\int_{0}^{t+}\int_{Z}F_{s}(z)N(dsdz),\int_{0}^{t+}\int_{Z}G_{s}(z)N(dsdz)\Big)\big]\\
&\leq C
E\big[\int_{0}^{t+}\int_{Z}H^{2}(F_{s}(z),G_{s}(z))N(dsdz)\big]\\
&=C
E\big[\int_{0}^{t}\int_{Z}H^{2}(F_{s}(z),G_{s}(z))ds\nu(dz)\big]
\end{split}
\end{equation}
where $C$ is the constant appearing in Theorem \ref{thm:000}.
\end{theorem}
\begin{proof}
Since
\begin{equation*}
\begin{split}
 H\left(F(s,z,\omega), G(s,z,\omega)\right)
&\leq H\left(F(s,z,\omega), \{0\}\right)+H\left(G(s,z,\omega),
\{0\}\right)\\
&=\|F(s,z,\omega)\|_{\bf k}+\|G(s,z,\omega)\|_{\bf k},
\end{split}
\end{equation*}
\begin{equation*}
\begin{split}
 H^{2}\left(F(s,z,\omega), G(s,z,\omega)\right)
&\leq \left(H\left(F(s,z,\omega), \{0\}\right)+H\left(G(s,z,\omega),
\{0\}\right)\right)^{2}\\
&\leq 2\|F(s,z,\omega)\|_{\bf k}^{2}+2\|G(s,z,\omega)\|_{\bf k}^{2},
\end{split}
\end{equation*}
and $F, G\in\mathscr{M}$, therefore both
$E\left[\int_{0}^{T+}\int_{Z}H\left(F(s,z,\omega),
G(s,z,\omega)N(dsdz)\right)\right]$ and \\
$E\left[\int_{0}^{T+}\int_{Z}H^{2}\left(F(s,z,\omega),
G(s,z,\omega)N(dsdz)\right)\right]$ are finite.
 By taking expectation on both sides of \eqref{eqn:110},
  immediately we obtain that
  $$
E\left[H\left(\int_{0}^{T+}\int_{Z}F_{s}(z)N(dsdz),
\int_{0}^{T+}\int_{Z}G_{s}(z)N(dsdz)\right)\right]<\infty
  $$
  and \eqref{eqn:111} holds. By Theorem \ref{th:squarebounded}, we have
  that
$$E\left[H^{2}\left(\int_{0}^{T+}\int_{Z}F_{s}(z)N(dsdz),
\int_{0}^{T+}\int_{Z}G_{s}(z)N(dsdz)\right)\right]$$ is
finite. Then by \eqref{eqn:110} and Theorem \ref{thm:000}, we have
  \begin{equation*}
  \begin{split}
  &E\left[H^{2}\left(\int_{0}^{t+}\int_{Z}F_{s}(z)N(dsdz),
\int_{0}^{t+}\int_{Z}G_{s}(z)N(dsdz)\right)\right]\\
&\leq E\left[\left(\int_{0}^{t+}\int_{Z}H\left(F_{s}(z),
G_{s}(z)\right)N(dsdz)\right)^{2}\right]\\
&\leq C E\left[\int_{0}^{t}\int_{Z}H^{2}\left(F_{s}(z),
G_{s}(z)\right)ds\nu(dz)\right]\\
&=C E\left[\int_{0}^{t+}\int_{Z}H^{2}\left(F_{s}(z),
G_{s}(z)\right)N(dsdz)\right],
  \end{split}
  \end{equation*}
  which implies \eqref{eqn:112}.
\end{proof}
\section{ Set-valued stochastic integral equation}
\label{author_sec:4}
In this section, we study the strong solution to a set-valued
stochastic integral  equation. Assume $\frak X$ is a separable
M-type 2 Banach space, $\mathcal F$ is separable with respect to
$P$. $(Z, \mathcal{B}(Z))$ is a separable Banach space with finite
measure $\nu$. Let the functions

 $a(\cdot, \cdot): [0, T]\times\bf K(\frak X)\rightarrow\bf K(\frak X)$ be
$\big(\mathcal B\big([0, T]\big)\otimes\sigma(\mathcal
C)\big)/\sigma(\mathcal C)$-measurable,

 $b(\cdot,\cdot): [0, T]\times\bf K(\frak X)\rightarrow\frak X$
 be
 $\big(\mathcal B\big([0, T]\big)\otimes\sigma(\mathcal C)\big)/\mathcal B(\frak
 X)$-measurable, and

$c(\cdot,\cdot,\cdot): [0, T]\times Z\times \bf K(\frak
X)\rightarrow \bf K(\frak X)$
 be
 $\big(\mathcal B\big([0, T]\big)\otimes\mathcal{B}(Z)\otimes\sigma(\mathcal C)\big)/\sigma(\mathcal
 C)$-measurable.

 Let $\{X_{t}: t\in[0, T]\}$ be a $\mathcal {P}$-predictable set-valued stochastic
process. Then $X: [0, T] \times \Omega\rightarrow \bf K(\frak X)$
can be considered as a $\mathcal{P}/\sigma(\mathcal C)$-measurable
function. By the property of composition of mappings, as a manner
similar to the proof of Lemma 4.1 in \cite{Zha3},  we can obtain
that:

(1). $a(t, X_{t}(\omega)): [0, T]\times\Omega\rightarrow \bf K(\frak
X)$ is $\mathcal {P}$-predictable;

(2). $ b(t, X_{t}(\omega)): [0, T]\times\Omega\rightarrow \frak X $
is $\mathcal {P}$-predictable;

(3). $ c(t, z, X_{t}(\omega)): [0, T]\times Z\times\Omega\rightarrow
\bf K(\frak X) $ is $\mathscr{S}$-predictable.

 Assume the above functions $a, b, c$
also satisfy
 the following conditions :
 \begin{equation}\label{eqn:linear}
 \| a(t, X)\|_{\bf K}+\|b(t, X)\|+\int_{Z}\|c(t,z,X)\|_{\bf K}\nu(dz)\leq C_{1}\big(1+\|X\|_{\bf K}\big),
 \end{equation}
for $X\in{\bf K}(\frak X), t\in[0, T]$ and some constant $C_{1}$ and
\begin{equation}\label{eqn:lip}
H^{2}\big(a(t, X),a(t, Y)\big)+\|b(t, X)-b(t,
Y)\|^{2}
+\int_{Z}H^{2}(c(t,z,X), c(t,z,Y))\nu(dz)\leq C_{2} H^{2}(X, Y),
\end{equation}
for $X, Y\in {\bf K}(\frak X)$, $t\in [0, T]$ and some constant
$C_{2}$.

Let $X_{0}$ be an $L^{2}$-integrably bounded set-valued
 random variable, $\{B_{t}:t\in[0,T]\}$ a real valued Brownian
 motion and $N_{\bf p}$ a stationary Poisson point process with
 characteristic measure $\nu$. It is reasonable to
define the set-valued stochastic integral equation as follows:
\begin{definition}
\begin{equation}\label{eqn:equation}
X_{t}=cl\big\{X_{0}+\int_{0}^{t}a(s,X_{s})ds+\int_{0}^{t}b(s,X_{s})dB_{s}
+\int_{0}^{t+}\int_{Z}c(s,z,X_{s-})N(dzds)\big\},
\end{equation}
for $t\in[0,T]$ a.s.

 Suppose that $\{X_{t}:
t\in[0, T]\}$ is an $\mathcal F_{t}$-adapted and measurable
set-valued process, which is right continuous in $t$ with respect to
$H$ almost surely. Then it is called a {\em strong solution} if it
satisfies the equation \eqref{eqn:equation}.
\end{definition}

\begin{remark} There are four terms on the right hand side of equation
\eqref{eqn:equation}. Every term is measurable and bounded a.s. Then
 the closure of the sum is measurable
and bounded a.s. Thus the right hand side of formulae \eqref{eqn:equation} makes sense.

 If the initial value is not
only $L^{2}$-integrably bounded but also weakly compact in $\frak
X$, then it is not necessary to take the closure in the right hand
side in \eqref{eqn:equation}(cf. (4.3) in \cite{Zha3}).
\end{remark}

\begin{theorem}\label{thm:exist}
 Assume that $\mathcal F$ is separable with respect to $P$. Let $T>0$, and let $a(\cdot,\cdot):[0, T]\times
{\bf K}(\frak X)\rightarrow\bf K(\frak X)$,
 $b(\cdot,\cdot): [0, T]\times {\bf K}(\frak X)\rightarrow\frak X$ and
  $c(\cdot,\cdot,\cdot): [0, T]\times Z\times {\bf K}(\frak X)\rightarrow {\bf K}(\frak X)$ be measurable
 functions satisfying conditions \eqref{eqn:linear} and
 \eqref{eqn:lip}. Then for any given $L^{2}$-integrably bounded initial value $X_{0}$,
 there exists a unique
strong solution to \eqref{eqn:equation}. The
 unique strong solution is
 right continuous in $t$ with respect to the Hausdorff metric. In the
 above, the uniqueness means $P\Big(H(X_{t},Y_{t})=0\ for \ all\
 t\in[0,T]\Big)=1$ for any strong solutions $X_{t}$ and $Y_{t}$ to
 \eqref{eqn:equation}.
\end{theorem}

\begin{proof}
 As a manner similar to that of solving single-valued stochastic differential equation,
  we use the successive approximation method to construct a
solution of equation \eqref{eqn:equation}.

Define $Y_{t}^{0}=X_{0}$, and
 $Y_{t}^{k}=Y_{t}^{k}(\omega)$ for $k\in\mathbb N$ inductively as follows:
\begin{equation}
Y_{t}^{k+1}=cl\Big\{X_{0}+\int_{0}^{t} a(s, Y^{k}_{s})ds
+\int_{0}^{t} b(s,
Y^{k}_{s})dB_{s}
+\int_{0}^{t+}\int_{Z}c(s,z,Y^{k}_{s-})N(dzds)\Big\}.
\end{equation}
By property of Hausdorff metric, we have
\begin{equation*}
\begin{split}
&H(Y_{t}^{k+1}, Y_{t}^{k})
\leq
H\Big(\int_{0}^{t}a(s,Y_{s}^{k})ds,\int_{0}^{t}a(s,Y_{s}^{k-1})ds\Big)
+\big\|\int_{0}^{t}\big(b(s,Y_{s}^{k})-b(s,Y_{s}^{k-1})\big)dB_{s}\big\|\\
&+H\Big(\int_{0}^{t+}\int_{Z}c(s,z,Y_{s-}^{k})N(dzds), \int_{0}^{t+}\int_{Z}
c(c,s,Y^{k-1}_{s-}))N(dzds)\Big).
\end{split}
\end{equation*}
\begin{equation*}
\begin{split}
&E\Big[\sup_{s\in[0,t]}H^{2}(Y_{s}^{k+1}, Y_{s}^{k})\Big]\\
\leq
3E\Big[&\sup_{s\in[0,t]}H^{2}\big(\int_{0}^{s}a(\tau,Y_{\tau}^{k})d\tau,
\int_{0}^{s}a(\tau,Y_{\tau}^{k-1})d\tau\big)
+\sup_{0\leq s\leq
t}\big\|\int_{0}^{s}b(\tau,Y_{\tau}^{k})dB_{\tau}-\int_{0}^{s}b(\tau,Y_{\tau}^{k-1})dB_{\tau}\big\|^{2}\\
&+\sup_{0<s\leq
t}H^{2}\Big(\int_{0}^{t+}\int_{Z}c(s,z,Y_{s-}^{k})N(dzds),
\int_{0}^{t+}\int_{Z} c(c,s,Y^{k-1}_{s-})N(dzds)\Big)\Big].
\end{split}
\end{equation*}
By condition \eqref{eqn:lip}
 and Doob maximal martingale inequality, we have
\begin{equation} \label{eqn:113}
E\Big[\sup_{0\leq s\leq
t}H^{2}\big(\int_{0}^{s}a(\tau,Y_{\tau}^{k})d\tau,
\int_{0}^{s}a(\tau,Y_{\tau}^{k-1})d\tau\big)\Big]\leq
TC_{2}E\Big[\int_{0}^{t}H^{2}(Y_{\tau}^{k},Y_{\tau}^{k-1})d\tau\Big]
\end{equation}
\begin{equation}\label{eqn:114}
 E\Big[\sup_{0\leq s\leq
t}\big\|\int_{0}^{s}b(\tau,Y_{\tau}^{k})dB_{\tau}-\int_{0}^{s}b(\tau,Y_{\tau}^{k-1})dB_{\tau}\big\|^{2}\Big]
 \leq
4C_{\frak
X}C_{2}E\Big[\int_{0}^{t}H^{2}(Y_{\tau}^{k},Y_{\tau}^{k-1})d\tau\Big].
\end{equation}
 By Theorem \ref{thm:000}, Lemma \ref{lem:inequality1}, Theorem \ref{thm:300} and condition
 \eqref{eqn:lip},
 \begin{equation*}
 \begin{split}
 &E\Big[\sup_{0<s\leq
t}H^{2}\Big(\int_{0}^{t+}\int_{Z}c(s,z,Y_{s-}^{k})N(dzds),
\int_{0}^{t+}\int_{Z} c(c,s,Y^{k-1}_{s-})N(dzds)\Big)\Big]\\
&\leq E\Big[\sup_{0<s\leq
t}\Big(\int_{0}^{t+}\int_{Z}H\big(c(s,z,Y_{s-}^{k}),
c(c,s,Y^{k-1}_{s-})\big)N(dzds)\Big)^{2}\Big]\\
&\leq CE\Big[\int_{0}^{t+}\int_{Z}H^{2}\big(c(s,z,Y_{s-}^{k}),
c(c,s,Y^{k-1}_{s-})\big)ds\nu(dz)\Big]\\
&=CE\Big[\int_{0}^{t+}\big(\int_{Z}H^{2}\big(c(s,z,Y_{s-}^{k}),
c(c,s,Y^{k-1}_{s-})\big)\nu(dz)\big)ds\Big]
\leq CC_{2}E\Big[\int_{0}^{t}H^{2}(Y^{k}_{\tau},
Y^{k-1}_{\tau})d\tau\Big].
 \end{split}
 \end{equation*}
 Therefore, we obtain
 \begin{equation*}
E\Big[\sup_{s\in[0,t]}H^{2}(Y_{s}^{k+1}, Y_{s}^{k})\Big]\leq
(3TC_{2}+12C_{\frak
X}C_{2}+3CC_{2})E\Big[\int_{0}^{t}H^{2}(Y^{k}_{\tau},
Y^{k-1}_{\tau})d\tau\Big]
 \end{equation*}
Setting $c:=9C_{2}(T\vee 4C_{\frak X}\vee C)$ and
$\bigtriangleup_{k}(t):=E\Big[\sup_{s\in[0,t]}H^{2}(Y_{s}^{k+1}, Y_{s}^{k})\Big],$
 then by induction, we have
 \begin{eqnarray*}
 &&\bigtriangleup_{k}(T)=E\Big[\sup_{s\in[0,T]}H^{2}(Y_{s}^{k+1},
 Y_{s}^{k})\Big]
 \leq
 c\int_{0}^{T}\bigtriangleup_{k-1}(\tau)d\tau\\
 &&\leq
 c^{k}\int_{0}^{T}\int_{0}^{\tau_{k-1}}\cdots \int_{0}^{\tau_{1}}\bigtriangleup_{0}(\tau_{0})d\tau_{0}\cdots d\tau_{k-2}d\tau
 \leq
 c^{k}\bigtriangleup_{0}(T)\int_{0}^{T}\int_{0}^{\tau_{k-1}}\cdots \int_{0}^{\tau_{1}}d\tau_{0}\cdots d\tau_{k-2}d\tau.
 \end{eqnarray*}
Hence, we obtain
$\bigtriangleup_{k}(T)\leq \frac{(cT)^{k}}{k!}\bigtriangleup_{0}(T).$
Therefore, the series $\sum_{k=1}^{\infty}\bigtriangleup_{k}(T)$
converges. Then
$$\sum_{k=1}^{\infty}\sup_{t\in[0,T]}H^{2}(Y^{k}_{t},Y^{k-1}_{t})<+\infty \ a.s.,$$
which implies the sequence $\{Y_{t}^{k}:k\in\mathbb N\}$ uniformly
(with respect to $t$) converges to a set-valued stochastic process
denoted by $\{Y_{t}: t\in[0,T]\}$ by the completeness of the space
$L^{2}\big(\Omega; ({\bf {K}}_{b}(\frak X), H)\big)$. Since both the integral of set-valued stochastic processes with respect to Lebesgue measure $t$ and the integral with respect to Brownian motion are continuous in $t$, together with Theorem \ref{thm:300}, we obtain that the process
$\{Y_{t}\}$ is right continuous in $t$ with respect to the Hausdorff
metric $H$ and satisfies \eqref{eqn:equation}.

Now we show the uniqueness of solutions. Assume there are two
solutions $\{X_{t}:t\in[0,T]\}$ and $\{Y_{t}:t\in[0,T]\}$ with the
same initial value $X_{0}$. Denote
$\bigtriangleup(t)=E\Big[\sup_{s\in[0,t]}H^{2}(X_{s},Y_{s})\Big].$
Then through the same way as above, we have
$\bigtriangleup(T)\leq \frac{(cT)^{k}}{k!}\bigtriangleup(T).$
Letting $k\rightarrow \infty$, we obtain $\bigtriangleup(T)=0$,
which implies $P\Big(H(X_{t},{Y_{t}})=0\ for \ all\
t\in[0,T]\Big)=1$.
\end{proof}
\section{Concluding remark}
\label{author_sec:5}
The main result of this paper is that the set-valued integral with respect to the compensated Poisson measure is not a martingale unless the integrand degenerates into a single-valued process. The proof uses the Hahn decomposition of a Banach space and bounded linear functionals. Since integrals with respect to Poisson point process are integrably bounded, the differential equation with set-valued jump makes sense.  Due to the complexity in real world, set-valued random variable is a good tool to model the uncertainty including both randomness and imprecision. We expect that the model \eqref{eqn:equation} has potential applications to practical fields. For instance, single-valued stochastic calculus has surprising applications in mathematical finance and dynamics \cite{Kunita2}. It is also reasonable to consider the price of finance derivative as an interval-valued random variable due to high frequency fluctuations and unseen events. Ogura \cite{Ogura} studied the set-valued  Black-Scholes equation. Sometimes there is a big change of price since some unusual and unpredictable causes. A possible model for this situation is set-valued stochastic differential equation with jump, which is a natural extension of the equation in \cite{Ogura}.  Another example of potential application is on detection of echo signal of a sea clutter, which is very important in the defense and civilian business. Due to the fluid dynamics, classical stochastic differential equation is used to modelling the echo signal's phase and amplitude (\cite{Ward}). The sea surface may have a big change during a very short period since the complex fluid dynamics or the sudden strong wind. It is reasonable to consider the amplitude of sea clutter as a set-valued process. The sharp change of sea surface can be described as a Poisson jump.

\vskip 0.6cm


\bibliographystyle{model1a-num-names}
\bibliography{<your-bib-database>}



\end{document}